\documentclass[a4paper]{article}

\usepackage{epsfig}
\usepackage{graphicx}
\usepackage{amsbsy}
\usepackage{amsmath}
\usepackage{amsfonts}
\usepackage{amssymb}
\usepackage{textcomp}
\usepackage{hyperref}
\usepackage{aliascnt}

\newcommand{\mcm}[3]{\newcommand{#1}[#2]{{\ensuremath{#3}}}} 

\mcm{\tuple}{1}{\langle #1 \rangle}
\mcm{\name}{1}{\ulcorner #1 \urcorner}
\mcm{\Nbb}{0}{\mathbb{N}}
\mcm{\Zbb}{0}{\mathbb{Z}}
\mcm{\Rbb}{0}{\mathbb{R}}
\mcm{\Cbb}{0}{\mathbb{C}}
\mcm{\Qbb}{0}{\mathbb{Q}}
\mcm{\Bcal}{0}{\cal B}
\mcm{\Ccal}{0}{\cal C}
\mcm{\Dcal}{0}{\cal D}
\mcm{\Ecal}{0}{\cal E}
\mcm{\Fcal}{0}{\cal F}
\mcm{\Gcal}{0}{\cal G}
\mcm{\Hcal}{0}{\cal H}
\mcm{\Ical}{0}{\cal I}
\mcm{\Jcal}{0}{\cal J}
\mcm{\Kcal}{0}{\cal K}
\mcm{\Lcal}{0}{\cal L}
\mcm{\Mcal}{0}{\cal M}
\mcm{\Ncal}{0}{\cal N}
\mcm{\Ocal}{0}{{\cal O}}
\mcm{\Pcal}{0}{{\cal P}}
\mcm{\Qcal}{0}{{\cal Q}}
\mcm{\Rcal}{0}{{\cal R}}
\mcm{\Scal}{0}{{\cal S}}
\mcm{\Tcal}{0}{{\cal T}}
\mcm{\Ucal}{0}{{\cal U}}
\mcm{\Vcal}{0}{{\cal V}}
\mcm{\Wcal}{0}{{\cal W}}
\mcm{\Xcal}{0}{{\cal X}}
\mcm{\Ycal}{0}{{\cal Y}}
\mcm{\Mfrak}{0}{\mathfrak M}

\mcm{\restric}{0}{\upharpoonright}
\mcm{\upset}{0}{\uparrow}
\mcm{\onto}{0}{\twoheadrightarrow}
\mcm{\smallNbb}{0}{{\small \mathbb{N}}}
\DeclareMathOperator{\preop}{op}
\mcm{\op}{0}{^{\preop}}

\newcommand{\se}{\subseteq}
{\begin{array}{c}
\setlength{\unitlength}{1em}}%
{\end{array}}

\usepackage{amsthm}

\newcommand{\theoremize}[2]{\newaliascnt{#1}{thm} \newtheorem{#1}[#1]{#2} \aliascntresetthe{#1}}

\theoremstyle{plain}
\newtheorem{thm}{Theorem}[section]
\theoremize{lem}{Lemma}
\theoremize{sublem}{Sublemma}
\theoremize{claim}{Claim}
\theoremize{rem}{Remark}
\theoremize{prop}{Proposition}
\theoremize{cor}{Corollary}
\theoremize{que}{Question}
\theoremize{oque}{Open Question}
\theoremize{con}{Conjecture}

\theoremstyle{definition}
\theoremize{dfn}{Definition}
\theoremize{eg}{Example}
\theoremize{exercise}{Exercise}
\theoremstyle{plain}

\usepackage{verbatim}
\usepackage[inline]{enumitem}
\usepackage[all]{xy}
\usepackage{xifthen}

\usepackage[usenames,dvipsnames]{pstricks}
\usepackage{epsfig}

\title{On tree-decompositions of one-ended graphs}

\author{
Johannes Carmesin%
\thanks{Department of Pure Mathematics and Mathematical Statistics, University of Cambridge, Wilberforce Road, Cambridge CB3 0WB, United Kingdom} , 
Florian Lehner%
\thanks{Mathematics Institute, University of Warwick, Zeeman Building, Coventry CV4 7AL, United Kingdom \newline Florian Lehner was supported by the Austrian Science Fund (FWF), grant J 3850-N32} , and 
R\"ognvaldur G.\ M\"oller%
\thanks{Science Institute, University of Iceland, IS-107 Reykjav\'ik, Iceland \newline R\"ognvaldur G.\ M\"oller acknowledges support from the University of Iceland Research Fund}
}

\newcommand{\sm}{\setminus}

\DeclareMathOperator{\Aut}{Aut}

\newcommand{\Srel}{\mathcal S_{\omega}}

\begin{document}
\maketitle

\begin{abstract}
A graph is one-ended if it contains a ray (a one way infinite path) and whenever we remove a finite number of vertices from the graph then what remains has only one component which contains rays.  A vertex $v$ {\em dominates} a ray in the end if there are infinitely many paths connecting $v$ to the ray such that any two of these paths have only the vertex $v$ in common.  
We prove that  if a one-ended graph contains no ray which is dominated by a vertex and no infinite family of pairwise disjoint rays, then it has a tree-decomposition such that the decomposition tree is one-ended
and the tree-decomposition is invariant under the group of automorphisms. 

This can be applied to prove a conjecture 
of Halin from 2000 that the automorphism group of such a graph cannot be countably infinite and solves a recent problem of Boutin and Imrich. Furthermore, it implies that every transitive one-ended graph contains an infinite family of pairwise disjoint rays.
\end{abstract}

\section{Introduction}

The ends of a graph $G$ are defined as equivalence classes of rays (one sided infinite paths).  Two rays are said to belong to the same end if for every finite set $F$ of vertices the same component of $G\setminus F$ contains infinitely many vertices from both rays.  The ends of a graph are a tool to capture the \lq\lq large-scale\rq\rq\ structure of an infinite graphs. In particular if a graph has more than one end then the graph can be said to be \lq\lq tree-like\rq\rq.  
In \cite{zbMATH06486846}, Dunwoody and Kr\"on constructed so called {\em structure trees} to describe this \lq\lq tree-likeness\rq\rq\ of a graph with more than one end.  A structure tree is constructed from a nested $\Aut(G)$-invariant family of separations of the graph such that the action of $\Aut(G)$ on the family of separations gives an action of $\Aut(G)$ on the tree.  The work of Dunwoody and Kr\"on builds on the book \cite{zbMATH00041228} by Dicks and Dunwoody, see also a recent account of this theory in \cite{zbMATH06787734}.

{\em Tree-decompositions} (defined in Section~\ref{sec:treedecomp}) are 
very similar to structure trees; and do play a central role in Graph Minor Theory and are a standard tool in Graph Theory to describe tree-structure \cite{zbMATH05760492}. If the nested set of separations of a tree-decomposition is $\Aut(G)$-invariant then the tree-decomposition is also $\Aut(G)$-invariant and $\Aut(G)$ acts on the decomposition tree.  

Dunwoody and Kr\"on 
apply their construction to obtain a combinatorial proof of generalization of Stalling's theorem of groups with at least two 
ends. This method has multifarious other applications, as demonstrated by Hamann in  \cite{zbMATH06105480} and Hamann and Hundertmark in \cite{zbMATH06152566}.

However, for graphs with only a single end, such as the 2-dimensional grid, these structure trees and the related tree-decompositions may be trivial. 
Hence such a structural understanding of this class of graphs remains elusive. 
If the end of a one-ended graph has finite vertex degree, that is, 
there is no infinite set of pairwise vertex-disjoint rays belonging to that end, then Halin showed in 1965 \cite{zbMATH03212734} that there are tree-decompositions displaying 
the end. A precise definition can be found towards the end of Section~\ref{sec:treedecomp}, but essentially this means that the tree also only has one end and that end \lq\lq corresponds\rq\rq to the end of the graph is a precise way. Nevertheless, for these tree-decompositions to be of any use for applications as above, 
one needs them to have the additional property that they are invariant under the group of 
automorphisms. 
Unfortunately such tree-decompositions do not exist for all graphs in question, see Example~\ref{undom_nec} below. Note that in this example there is a vertex $v$ dominating the end, that is, for every ray in the end there are infinitely many paths connecting $v$ to the ray such that any two of these paths have only the vertex $v$ in common.
In this paper we construct such tree-decompositions if the end is not dominated.

\begin{thm}\label{main:intro}
Every one-ended graph whose end is undominated and has finite vertex degree has a tree-decomposition that displays its 
end and that is invariant under the group of automorphisms.
\end{thm}

A very simple example is shown in Figure~\ref{fig2}.

 This better structural understanding leads to applications similar to those for graphs with more than 
one end. Indeed, below we deduce from Theorem~\ref{main:intro} a conjecture of Halin from 2000, and 
answer a recent question of Boutin and Imrich. A further application was pointed out by 
Hamann.



\begin{figure}
\centering
%
%
\psscalebox{1.0 1.0} 
{
\begin{pspicture}(0,-2.6333702)(9.877313,2.6333702)
\psline[linecolor=black, linewidth=0.02](0.06899033,-2.5643797)(0.46899033,-2.5643797)
\psline[linecolor=black, linewidth=0.02](0.8689903,-2.5643797)(1.2689903,-2.5643797)
\psline[linecolor=black, linewidth=0.02](1.6689904,-2.5643797)(2.0689902,-2.5643797)
\psline[linecolor=black, linewidth=0.02](2.4689903,-2.5643797)(2.8689904,-2.5643797)
\psline[linecolor=black, linewidth=0.02](0.46899033,-1.9643795)(0.8689903,-1.9643795)
\psline[linecolor=black, linewidth=0.02](2.0689902,-1.9643795)(2.4689903,-1.9643795)
\psline[linecolor=black, linewidth=0.02](0.06899033,-2.5643797)(0.46899033,-1.9643795)
\psline[linecolor=black, linewidth=0.02](0.46899033,-2.5643797)(0.8689903,-1.9643795)
\psline[linecolor=black, linewidth=0.02](0.8689903,-2.5643797)(0.46899033,-1.9643795)
\psline[linecolor=black, linewidth=0.02](1.2689903,-2.5643797)(0.8689903,-1.9643795)
\psline[linecolor=black, linewidth=0.02](0.46899033,-1.9643795)(1.2689903,-1.1643796)
\psline[linecolor=black, linewidth=0.02](1.6689904,-2.5643797)(2.0689902,-1.9643795)
\psline[linecolor=black, linewidth=0.02](2.0689902,-2.5643797)(2.4689903,-1.9643795)
\psline[linecolor=black, linewidth=0.02](2.4689903,-2.5643797)(2.0689902,-1.9643795)
\psline[linecolor=black, linewidth=0.02](2.8689904,-2.5643797)(2.4689903,-1.9643795)
\psline[linecolor=black, linewidth=0.02](3.2689903,-2.5643797)(3.6689904,-2.5643797)
\psline[linecolor=black, linewidth=0.02](4.06899,-2.5643797)(4.4689903,-2.5643797)
\psline[linecolor=black, linewidth=0.02](4.8689904,-2.5643797)(5.2689905,-2.5643797)
\psline[linecolor=black, linewidth=0.02](5.66899,-2.5643797)(6.06899,-2.5643797)
\psline[linecolor=black, linewidth=0.02](3.6689904,-1.9643795)(4.06899,-1.9643795)
\psline[linecolor=black, linewidth=0.02](5.2689905,-1.9643795)(5.66899,-1.9643795)
\psline[linecolor=black, linewidth=0.02](1.2689903,-1.1643796)(1.6689904,-1.1643796)
\psline[linecolor=black, linewidth=0.02](4.4689903,-1.1643796)(4.8689904,-1.1643796)
\psline[linecolor=black, linewidth=0.02](2.8689904,-0.16437958)(3.2689903,-0.16437958)
\psline[linecolor=black, linewidth=0.02](0.8689903,-1.9643795)(1.6689904,-1.1643796)
\psline[linecolor=black, linewidth=0.02](1.6689904,-1.1643796)(2.4689903,-1.9643795)
\psline[linecolor=black, linewidth=0.02](2.0689902,-1.9643795)(1.2689903,-1.1643796)
\psline[linecolor=black, linewidth=0.02](1.2689903,-1.1643796)(2.8689904,-0.16437958)
\psline[linecolor=black, linewidth=0.02](1.6689904,-1.1643796)(3.2689903,-0.16437958)
\psline[linecolor=black, linewidth=0.02](3.2689903,-2.5643797)(3.6689904,-1.9643795)
\psline[linecolor=black, linewidth=0.02](3.6689904,-2.5643797)(4.06899,-1.9643795)
\psline[linecolor=black, linewidth=0.02](3.6689904,-1.9643795)(4.06899,-2.5643797)
\psline[linecolor=black, linewidth=0.02](4.06899,-1.9643795)(4.4689903,-2.5643797)
\psline[linecolor=black, linewidth=0.02](4.8689904,-2.5643797)(5.2689905,-1.9643795)
\psline[linecolor=black, linewidth=0.02](5.2689905,-2.5643797)(5.66899,-1.9643795)
\psline[linecolor=black, linewidth=0.02](5.2689905,-1.9643795)(5.66899,-2.5643797)
\psline[linecolor=black, linewidth=0.02](5.66899,-1.9643795)(6.06899,-2.5643797)
\psline[linecolor=black, linewidth=0.02](3.6689904,-1.9643795)(4.4689903,-1.1643796)
\psline[linecolor=black, linewidth=0.02](4.06899,-1.9643795)(4.8689904,-1.1643796)
\psline[linecolor=black, linewidth=0.02](4.4689903,-1.1643796)(5.2689905,-1.9643795)
\psline[linecolor=black, linewidth=0.02](4.8689904,-1.1643796)(5.66899,-1.9643795)
\psline[linecolor=black, linewidth=0.02](4.4689903,-1.1643796)(2.8689904,-0.16437958)
\psline[linecolor=black, linewidth=0.02](3.2689903,-0.16437958)(4.8689904,-1.1643796)
\psdots[linecolor=black, dotsize=0.14](0.06899033,-2.5643797)
\psdots[linecolor=black, dotsize=0.14](0.46899033,-2.5643797)
\psdots[linecolor=black, dotsize=0.14](0.8689903,-2.5643797)
\psdots[linecolor=black, dotsize=0.14](1.2689903,-2.5643797)
\psdots[linecolor=black, dotsize=0.14](0.46899033,-1.9643795)
\psdots[linecolor=black, dotsize=0.14](0.8689903,-1.9643795)
\psdots[linecolor=black, dotsize=0.14](1.2689903,-1.1643796)
\psdots[linecolor=black, dotsize=0.14](1.6689904,-1.1643796)
\psdots[linecolor=black, dotsize=0.14](2.0689902,-1.9643795)
\psdots[linecolor=black, dotsize=0.14](2.4689903,-1.9643795)
\psdots[linecolor=black, dotsize=0.14](1.6689904,-2.5643797)
\psdots[linecolor=black, dotsize=0.14](2.0689902,-2.5643797)
\psdots[linecolor=black, dotsize=0.14](2.4689903,-2.5643797)
\psdots[linecolor=black, dotsize=0.14](2.8689904,-2.5643797)
\psdots[linecolor=black, dotsize=0.14](3.2689903,-2.5643797)
\psdots[linecolor=black, dotsize=0.14](3.6689904,-2.5643797)
\psdots[linecolor=black, dotsize=0.14](3.6689904,-1.9643795)
\psdots[linecolor=black, dotsize=0.14](4.06899,-1.9643795)
\psdots[linecolor=black, dotsize=0.14](4.06899,-2.5643797)
\psdots[linecolor=black, dotsize=0.14](4.4689903,-2.5643797)
\psdots[linecolor=black, dotsize=0.14](4.8689904,-2.5643797)
\psdots[linecolor=black, dotsize=0.14](5.2689905,-1.9643795)
\psdots[linecolor=black, dotsize=0.14](5.66899,-1.9643795)
\psdots[linecolor=black, dotsize=0.14](5.2689905,-2.5643797)
\psdots[linecolor=black, dotsize=0.14](5.66899,-2.5643797)
\psdots[linecolor=black, dotsize=0.14](6.06899,-2.5643797)
\psdots[linecolor=black, dotsize=0.14](4.4689903,-1.1643796)
\psdots[linecolor=black, dotsize=0.14](4.8689904,-1.1643796)
\psdots[linecolor=black, dotsize=0.14](2.8689904,-0.16437958)
\psdots[linecolor=black, dotsize=0.14](3.2689903,-0.16437958)
\psline[linecolor=black, linewidth=0.02](2.8689904,-0.16437958)(5.8689904,1.0356205)
\psline[linecolor=black, linewidth=0.02](3.2689903,-0.16437958)(6.2689905,1.0356205)
\psdots[linecolor=black, dotsize=0.14](5.8689904,1.0356205)
\psdots[linecolor=black, dotsize=0.14](6.2689905,1.0356205)
\psline[linecolor=black, linewidth=0.02](5.8689904,1.0356205)(6.2689905,1.0356205)
\psline[linecolor=black, linewidth=0.02](5.8689904,1.0356205)(7.66899,2.0356205)
\psline[linecolor=black, linewidth=0.02](6.2689905,1.0356205)(8.068991,2.0356205)
\psline[linecolor=black, linewidth=0.02](8.068991,0.03562042)(8.46899,0.03562042)
\psdots[linecolor=black, dotsize=0.14](8.068991,0.03562042)
\psdots[linecolor=black, dotsize=0.14](8.46899,0.03562042)
\psline[linecolor=black, linewidth=0.02](5.8689904,1.0356205)(8.068991,0.03562042)
\psline[linecolor=black, linewidth=0.02](6.2689905,1.0356205)(8.46899,0.03562042)
\rput{27.645975}(1.9007502,-3.4538467){\psframe[linecolor=white, linewidth=0.02, fillstyle=solid, dimen=outer](8.46899,2.4356203)(7.4689903,1.8356204)}
\psdots[linecolor=black, dotsize=0.06](7.8689904,2.0356205)
\psdots[linecolor=black, dotsize=0.06](8.2689905,2.2356205)
\psdots[linecolor=black, dotsize=0.06](8.068991,2.1356204)
\psdots[linecolor=black, dotsize=0.14](7.06899,-1.1643796)
\psdots[linecolor=black, dotsize=0.14](7.4689903,-1.1643796)
\psdots[linecolor=black, dotsize=0.14](9.068991,-1.1643796)
\psdots[linecolor=black, dotsize=0.14](9.46899,-1.1643796)
\psline[linecolor=black, linewidth=0.02](7.06899,-1.1643796)(8.068991,0.03562042)
\psline[linecolor=black, linewidth=0.02](7.4689903,-1.1643796)(8.46899,0.03562042)
\psline[linecolor=black, linewidth=0.02](8.46899,0.03562042)(9.46899,-1.1643796)
\psline[linecolor=black, linewidth=0.02](9.068991,-1.1643796)(8.068991,0.03562042)
\psline[linecolor=black, linewidth=0.02, linestyle=dashed, dash=0.17638889cm 0.10583334cm](7.06899,-1.1643796)(6.66899,-1.7643796)
\psline[linecolor=black, linewidth=0.02, linestyle=dashed, dash=0.17638889cm 0.10583334cm](7.4689903,-1.1643796)(7.06899,-1.7643796)
\psline[linecolor=black, linewidth=0.02, linestyle=dashed, dash=0.17638889cm 0.10583334cm](7.06899,-1.1643796)(7.4689903,-1.7643796)
\psline[linecolor=black, linewidth=0.02, linestyle=dashed, dash=0.17638889cm 0.10583334cm](7.4689903,-1.1643796)(7.8689904,-1.7643796)
\psline[linecolor=black, linewidth=0.02](7.06899,-1.1643796)(7.4689903,-1.1643796)
\psline[linecolor=black, linewidth=0.02, linestyle=dashed, dash=0.17638889cm 0.10583334cm](9.068991,-1.1643796)(8.66899,-1.7643796)
\psline[linecolor=black, linewidth=0.02, linestyle=dashed, dash=0.17638889cm 0.10583334cm](9.46899,-1.1643796)(9.068991,-1.7643796)
\psline[linecolor=black, linewidth=0.02, linestyle=dashed, dash=0.17638889cm 0.10583334cm](9.068991,-1.1643796)(9.46899,-1.7643796)
\psline[linecolor=black, linewidth=0.02, linestyle=dashed, dash=0.17638889cm 0.10583334cm](9.46899,-1.1643796)(9.86899,-1.7643796)
\psline[linecolor=black, linewidth=0.02](9.068991,-1.1643796)(9.46899,-1.1643796)
\end{pspicture}
}
\caption{The product of the canopy tree with $K_1$. This graph has a tree decomposition whose decomposition tree is the canopy tree. }\label{fig2}
\end{figure}
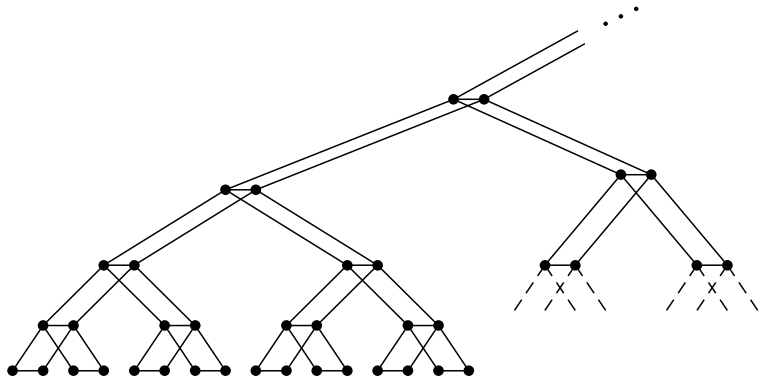

\vspace{.3cm}

{\bf Applications.}
In \cite{zbMATH01701677} Halin showed that one-ended graphs with vertex degree equal to one 
cannot have countably infinite automorphism group. Not completely satisfied with his result, he 
conjectured that this extends to one-ended graphs with finite vertex degree. Theorem 
\ref{main:intro} implies this conjecture.

\begin{thm}
\label{dichotomy}
Given a graph with one end which has finite vertex degree, its automorphism group is either finite 
or 
has at least $2^{\aleph_0}$ many elements.
\end{thm}

\vspace{.3cm}

Theorem~\ref{dichotomy} can be further applied to answer a question posed by Boutin and 
Imrich, who asked in~\cite{imrichboutin} whether there is a locally finite graph with only one end and linear growth and countably infinite automorphism group. Theorem~\ref{dichotomy} implies a negative answer to this question as well as strengthenings of further results of Boutin and Imrich, see Section~\ref{sec:dichotomy} for details.

\vspace{.3cm}

Finally, Matthias Hamann\footnote{personal communication} pointed out the following consequence of 
Theorem~\ref{main:intro}.

\begin{thm}\label{thm:enddegree}
The end of a transitive one-ended graph must have infinite vertex degree.
\end{thm}

This extends a result of Thomassen for locally finite graphs, see \cite[Proposition 5.6]{zbMATH00130612}.
We actually prove a stronger version of Theorem~\ref{thm:enddegree}, see Theorem~\ref{thm:oneend-qt-thick}, with 
`quasi-transitive'\footnote{Here a graph is \emph{quasi-transitive}, if there are 
only finitely many orbits of vertices under the automorphism group.} in place of `transitive'.

\vspace{.3cm}

The rest of this paper is structured as follows: in Section~\ref{sec:prelim} we set up all necessary 
notations and definitions. As explained in \cite{zbMATH06330529}, there is a close relation between 
tree-decompositions and 
nested sets of separations. In this paper we work mainly with nested sets of separations. 
In Section~\ref{sec:treedecomp} we prove Theorems~\ref{thm:aut_G_inv_td} and~\ref{main_td} that imply Theorem~\ref{main:intro}, and 
Section~\ref{sec:dichotomy} is devoted to the proof of Theorem~\ref{dichotomy}, and its implications on the work of Boutin and Imrich.
Finally, in Section~\ref{sec:qt-ends} we prove Theorem~\ref{thm:oneend-qt-thick} that implies Theorem~\ref{thm:enddegree}.  

Many of the lemmas we apply in this work were first proved by Halin. Since in some cases we need slight 
variants of the original results and also since Halin's original papers might not be easily 
accessible, proofs of some of these results are included in appendices.

\section{Preliminarlies}
\label{sec:prelim}

Throughout this paper $V(G)$ and $E(G)$ denote the sets of vertices and edges of a graph $G$, 
respectively. We refer to \cite{zbMATH05760492} for all graph theoretic notions which are not 
explicitly defined.

\subsection{Separations, rays and ends}

A \emph{separator} in a graph $G$ is a subset $S \subseteq V(G)$ such that $G-S$ is not connected. 
We say that a separator $S$ \emph{separates vertices $u$ and $v$} if $u$ and $v$ are in different 
components of $G-S$. Given two vertices $u$ and $v$, a separator $S$ separates $u$ and $v$ 
\emph{minimally} if 
it separates $u$ and $v$ and 
the components of $G-S$ containing $u$ and $v$ both have the whole of $S$ in their neighbourhood. 
The following lemma  can be found in Halin's 1965 paper {\cite[Statement 
2.4]{zbMATH00038800}, and also in his later paper \cite[Corollary 1]{zbMATH00089894} and then with a 
different proof.

\begin{lem}\label{halin}
Given vertices $u$ and $v$ and $k\in \Nbb$, there are only finitely many distinct separators of size 
at most $k$ separating $u$ and $v$ minimally.
\end{lem}

A \emph{separation} is a pair $(A,B)$ of subsets of $V(G)$ such that $A \cup B = V(G)$
 and there is no edge connecting $A \sm B$ to $B \sm A$.  This immediately implies that  if $u$ and 
$v$ are adjacent vertices in $G$ then $u$ and $v$ are both contained in either $A$ or $B$.  The 
sets $A$ 
and $B$ are called the \emph{sides} of the separation $(A,B)$. A separation $(A,B)$ is said to be 
\emph{proper} if both $A \sm B$ to $B \sm A$ are non-empty and then $A\cap B$ is a separator. 
A separation  $(A,B)$ is \emph{tight} if every vertex in $A\cap B$ has neighbours in both 
$A\setminus B$ and $B\setminus A$.
 The \emph{order} of a separation is the number of vertices in $A\cap B$. Throughout this paper we 
will only consider separations of finite order. 
The following is well-known. 

\begin{lem}\label{mod}{\rm (See \cite[Lemma 2.1]{zbMATH06427103})} Given any two separations $(A,B)$ 
and $(C,D)$  of $G$ then the sum of the orders of the separations $(A\cap C, B\cup D)$ and $(B\cap 
D, A\cup C)$ is equal to the sum of the orders of the separations $(A,B)$ and $(C,D)$.  In 
particular if the orders of $(A,B)$ and $(C,D)$ are both equal to $k$ then the sum of the orders of 
$(A\cap C, C\cup D)$ and $(B\cap D, A\cup C)$ is equal to $2k$.
\qed
\end{lem}

The separations $(A,B)$ and $(C,D)$ are \emph{strongly nested} if $A \se C$ and $D \se B$.
They are \emph{nested} if they are strongly nested after possibly exchanging `$(A,B)$' by `$(B,A)$' 
or  `$(C,D)$' by `$(D,C)$'. That is, $(A,B)$ and $(C,D)$ are nested if one 
of the following holds:
\begin{itemize}
\item $A \se C$ and $D \se B$,
\item $A \se D$ and $C \se B$,
\item $B \se C$ and $D \se A$,
\item $B \se D$ and $C \se A$.
\end{itemize}
We say a set $\mathcal S$ of separations is \emph{nested}, if any two separations in it are nested.

A \emph{ray} in a graph $G$ is a one-sided infinite path $v_0, v_1, \ldots$  in $G$.   The sub-rays 
of a ray are called its \emph{tails}. 
Given a finite separator $S$ of $G$, there is for every ray $\gamma$ a unique component of $G-S$ 
that contains all but finitely many vertices of $\gamma$.
We say that $\gamma$ \emph{lies} in that component of $G-S$.  Given a separation $(A,B)$ of finite 
order one can similarly say that $\gamma$ lies in one of the sides of the separation.
Two rays are \emph{in the same end} if they lie in the same component of $G-S$ for every finite 
separator of $G$. Clearly, this is an equivalence relation.
An equivalence class is called a \emph{(vertex) end}\footnote{A notion related to `vertex ends' 
are `topological ends'. In this paper we are mostly interested in graphs where no vertex 
dominates a vertex end. In this context the two notions of end agree.}. An alternative way 
to define ends is to say that two rays $R_1$ and $R_2$ are in the same end if there are infinitely 
many pairwise disjoint $R_1-R_2$ paths.   (Given subsets $X$ and $Y$ of the vertex set, an $X-Y$ 
path is a path that has its initial vertex in $X$ and terminal vertex in $Y$ and every other vertex 
is neither in $X$ nor $Y$.  In the case where $X=\{x\}$ then we speak of $x-Y$ paths instead of 
$X-Y$ paths and if $Y=\{y\}$ we speak of $x-y$ paths.)
An end $\omega$ \emph{lies} in a component $C$ of $G-S$  if every ray that belongs to $\omega$ lies in $C$. Clearly, every end lies in a unique component of $G-S$ for every finite 
separator $S$ and if $(A,B)$ is a separation of finite order then an end either lies in $A$ or $B$. 

A vertex $v \in V(G)$ \emph{dominates} an end $\omega$ of $G$, if there is no separation $(A,B)$ of 
finite order such that $v \in A\setminus B$ and $\omega$ lies in $B$. Equivalently, $v$ dominates 
$\omega$ if for every ray $R$ in $\omega$ there are infinitely many paths connecting $v$ to $R$ such 
that any two of them only intersect in $v$.

The \emph{vertex degree} of an end $\omega$ is equal to a natural number $k$ if the maximal 
cardinality of a family of pairwise disjoint rays belonging to the end is $k$.  If no such number 
$k$ exists then we say that the vertex-degree of the end is infinite. 
Halin \cite{zbMATH03212734}  (see also \cite[Theorem~8.2.5]{zbMATH05760492}) proved that if the 
vertex-degree of an end is infinite then there is an infinite family of pairwise disjoint rays 
belonging to the end.  Ends with finite vertex degree are sometimes called \emph{thin} and those with infinite vertex degree are called \emph{thick}.

The following lemma is well-known.  A proof can be found in Appendix A.

\begin{lem}\label{lem:finite_dominating}{\rm (Cf.~\cite[Section 3]{zbMATH01701677})}
Let $G$ be a connected graph and $\omega$ an end of $G$ having a finite vertex degree.   Then there 
are only finitely many vertices in $G$ that dominate the end $\omega$.  
\end{lem}

In this paper we are focusing on 1-ended graphs where the end $\omega$ has vertex degree $k$.  In 
the following definition we pick out a class of separations that are relevant in this case.

\begin{dfn}
Let $G$ be an arbitrary graph.  If  $\omega$ is an end of $G$ that has vertex degree $k$ then say 
that a separation $(A,B)$ is \emph{ $\omega$-relevant} if it has the following properties
\begin{itemize}
\item the order of $(A,B)$ is exactly $k$,
\item $A \sm B$ is connected,
\item every vertex in $A\cap B$ has a neighbour in $A \sm B$,
\item $\omega$ lies in $B$, and 
\item there is no separation $(C,D)$ of order $<k$ such that $A\se C$ and $\omega$ lies in $D$.
\end{itemize} 
Define $\Srel$ as the set of all $\omega$-relevant separations. 
\end{dfn}

 The following characterization of $\omega$-relevant separations is a Menger type result.  
A proof based on \cite{zbMATH03205933}  and \cite{zbMATH03212734} is contained in Appendix A.

\begin{lem}\label{Lrelevant}
Let $G$ be an arbitrary graph.  Suppose  $\omega$ is an end of $G$ with vertex degree $k$.  
\begin{enumerate}
\item  If $(A,B)$ is an $\omega$-relevant separation then there is a family of $k$ pairwise disjoint 
rays in $\omega$ such that each of them has its initial vertex in $A\cap B$.  

\item Conversely, if $(A,B)$ is a separation of order $k$ such that $A\sm B$ is connected, every 
vertex in $A\cap B$ has a neighbour in $A\sm B$, the end $\omega$ lies in $B$ and there is a family 
of $k$ disjoint rays in $\omega$ such that each of these rays has its initial vertex in $A\cap B$ 
then the separation $(A,B)$ is $\omega$-relevant.
\end{enumerate}
\end{lem}

In particular, for $(A,B) \in \Srel$ the component of $G-(A\cap B)$ in which $\omega$ lies has the 
whole of $A\cap B$ in its neighbourhood and hence every separation in $\Srel$ is tight. 
Note that the set $A\sm B$ completely determines the $\omega$-relevant separation $(A,B)$. 

The relation
\[
(A,B)\leq (C,D) :\iff A \se C  \mbox{ and } B\supseteq D
\]
defines a partial order on the set of all separations, so in particular on the set $\Srel$.  Since 
$(C,D)$ is a tight separation, the condition 
$A\se C$ implies that $D\se B$.   This is shown in \cite[(7) on p.~17]{zbMATH06330529} and the 
argument goes as follows:  Suppose that $D\not\se B$ and $x\in D\sm B$.  Then  $x\in A\se C$ so 
$x\in (C\cap D)\setminus B$.  Because $(C, D)$ is a tight separation, $x$ has a neighbour $y\in 
D\setminus C$.  But $x\in A\setminus B$ and hence $y$ must also be in $A$.  But $y\not\in C$, 
contradicting the assumption that $A\se C$.  Hence $D\subseteq B$ and  $(A,B)\leq (C,D) \iff A \se 
C$.   

The next result follow from results of Halin in \cite{zbMATH03212734}.  These results are in turn 
proved by using Menger's Theorem.  For the convenience of the reader a detailed proof is provided in 
Appendix A.   

\begin{thm}\label{Tsequence}   Let $G$ be a connected 1-ended graph such that the end $\omega$ is 
undominated and has finite vertex degree $k$.  
  Then there is a sequence $\{(A_n, B_n)\}_{n\geq 0}$ of $\omega$-relevant separations, such that 
the sequence of sets $B_n$ is strictly decreasing and for every finite set of vertices $F$ there is 
a number $n$ such that $F\subseteq A_n\sm B_n$.
\end{thm}

We will not use the following in our proof. 

\begin{rem}
Theorem~\ref{Tsequence} is also true if we leave out the assumption that $G$ is one-ended (and replace `the end $\omega$' by `there exists an end $\omega$ that').
\end{rem}

\subsection{Automorphism groups}\label{sec:autogroup}

An \emph{automorphism} of a graph $G = (V,E)$ is a bijective function $\gamma 
: V \to V$ that preserves adjacency and whose inverse also preserves adjacency. Clearly an automorphism $\gamma$ also induces a 
bijection $E \to E$ which by abuse of notation we will also call $\gamma$. The \emph{automorphism 
group of $G$}, i.e.\ the group of all automorphisms of $G$, will be denoted by $\Aut(G)$.

Let $\Gamma$ be a subgroup of $\Aut(G)$.  For a set $D\subseteq V(G)$ we define the \emph{setwise 
stabiliser} of $D$ as the subgroup $\Gamma_{\{D\}}=\{\gamma\in \Gamma\mid \gamma(D)=D\}$  and the 
\emph{pointwise stabiliser} of $D$ is defined as $\Gamma_{(D)}=\{\gamma\in \Gamma\mid 
\gamma(d)=d\mbox{ for all } d\in D\}$.  The setwise stabiliser is the subgroup of all elements in 
$\Gamma$ that leave the set $D$ invariant and the pointwise stabiliser is the subgroup of all those 
elements in $\Gamma$ that fix every vertex in $D$.  If $D\subseteq V(G)$ is invariant under $\Gamma$ 
then we use $\Gamma^D$ to denote the permutation group on $D$ induced by $\Gamma$, i.e.~$\Gamma^D$ 
is the group of all permutation $\sigma$ of $D$ such that there is some element $\gamma\in \Gamma$ 
such that the restriction of $\gamma$ to $D$ is equal to $\sigma$.  Note that $\Gamma_{(D)}$ is a 
normal subgroup of $\Gamma_{\{D\}}$ and the index $\Gamma_{(D)}$ in $\Gamma_{\{D\}}$ is equal to the 
number of elements in $(\Gamma_{\{D\}})^D$.

The full automorphism group of a graph has a special property relating to separations.  Suppose 
$\gamma$ is an automorphism of a graph $G$ and that $\gamma$ leaves both sides of a separation 
$(A,B)$ invariant and fixes every vertex in the separator $A\cap B$.  Then  the full automorphism 
group contains automorphisms $\sigma_A$ and $\sigma_B$ such that $\sigma_A$ like $\gamma$ on $A$ 
fixes every vertex in $B$ and \emph{vice versa} for $\sigma_B$.  Informally one can describe this 
property by saying that the pointwise stabiliser (in the full automorphism group) of a set $D$ of 
vertices acts indpendently on the components of $G-D$.   We will refer to this property as \emph{the 
independence property}.

There is a natural topology on $\Aut(G)$, called the \emph{permutation topology}: endow the vertex 
set with the discrete topology and consider the topology of pointwise convergence on $\Aut(G)$. 
Clearly, the permutation topology also makes sense for any group of permutations of a set. The 
following lemma is a special case of a result in \cite[(2.6) on p.~28]{zbMATH00044603}.  In 
particular it tells us that the limit of a sequence of automorphisms again is an automorphism. This 
fact will be central to the proof of Theorem~\ref{dichotomy}.
\begin{lem}
\label{aut-closed}
The automorphism group of a graph is closed in the set of all permutations of the vertex set endowed 
with the topology of pointwise convergence.
\end{lem}

The next result is also a special case of a result from Cameron's book refered to above.  This time 
we look at \cite[(2.2) on p.~28]{zbMATH00044603}.
 
\begin{lem}
\label{uncountable}
The automorphism group of a countable graph is finite, countably infinite or has at least 
$2^{\aleph_0}$ elements.
\end{lem}

\section{Invariant nested sets}
\label{sec:treedecomp}

In this section we will prove Theorems~\ref{thm:aut_G_inv_td} and~\ref{main_td}.  Theorem~\ref{main:intro} follows from Theorem~\ref{main_td}.  The following two facts about sequences of 
nested separations will be useful at several points in the proofs.

\begin{lem}
\label{infinitechains}  Let $G$ be a connected  graph.  
Assume that $(A_i,B_i)_{i \in \mathbb N}$ is a sequence of proper separations of order at most some 
fixed natural number $k$.   Assume also that $A_i \subsetneq A_{i-1}$, 
every $A_i \sm B_i$ is connected, and every vertex in $A_i \cap B_i$ has a neighbour in $A_i \sm 
B_i$. 
Define $X$ as the set of vertices contained in infinitely many $A_i$. Then
\begin{enumerate}
\item $X \subseteq B_i$ for all but finitely many $i$,
\item there is a unique end $\mu$ which lies in every $A_i$, and 
\item $x \in X$ if and only if $x$ dominates $\mu$.
\end{enumerate}
\end{lem}

\begin{proof}
First observe that $X = \bigcap_{i \in \mathbb N} A_i$ because the sequence $A_i$ is decreasing. Let 
$X'$ be the set of vertices in $X$ with a neighbour outside of $X$. For every $x \in X'$ we can find 
a neighbour $y$ of $x$ and $i_0 \in \mathbb N$ such that $y \notin A_i$ for every $i \geq i_0$. 
Since the edge $xy$ must be contained in either $A_i$ or $B_i$ we conclude that $x \in B_i$ and thus 
$x \in A_i \cap B_i$ for $i \geq i_0$. 

Hence there is $i_1\in \mathbb N$ such that $X' \subseteq A_i \cap B_i$ for every $i \geq i_1$.  The 
order of each separation is at most $k$, so $X'$ contains at most $k$ vertices. Now for $i \geq i_1$ 
every path from $X \sm B_i$ to $A_i \sm (X \cup B_i)$ must pass through $X'$ and thus through $B_i$. 
Since $A_i \sm B_i$ is connected this means that one of the two sets must be empty, i.e., either $X 
\sm B_i = \emptyset$ or $X\sm B_i=A_i\sm B_i$. Assume that the latter is the case. Then $A_i$ 
contains at most $k$ vertices which are not contained in $X$ and the same is clearly true for every 
$A_j$ for $j>i$. This contradicts the fact that the sequence $A_i$ was assumed to be infinite and 
strictly decreasing. We conclude that $X \subseteq B_i$ for $i \geq i_1$.  Note that this implies 
that $X=X'$ because if $i\geq i_1$ then $X\subseteq A_i\cap B_i$ and every vertex in $A_i\cap B_i$ 
has an neighbour in $A_i\setminus B_i$.

To see that there is an end $\mu$ which lies in every $A_i$ we construct a ray which has a tail 
in each $A_i$. For this purpose pick for $i\geq i_1$ a vertex $v_i \in A_i \sm X$ and paths $P_i$ 
connecting $v_i$ to $v_{i+1}$ in $A_i \sm X$. This is possible because $A_i \sm X$ contains $A_{i} 
\sm B_i$ and is connected ($A_i \sm B_i$ is connected and every vertex in $B_i \cap A_i$ has a 
neighbour in $A_i \sm B_i$). No vertex lies on infinitely many paths $P_i$ because no vertex is 
contained in infinitely many sets $A_i \sm X$. Hence the union of the paths $P_i$ is an infinite, 
locally finite graph and thus contains a ray. This ray belongs to an end $\mu$ which lies in 
every $A_i$.   

Finally we need to show that every vertex in $X$ dominates the end $\mu$.  Without loss of 
generality we can assume that $X \subseteq B_i$ for all $i$.  So,  let $R$ be a ray in $\mu$ and 
$x \in X$.  We will inductively construct infinitely many paths from $x$ to $R$ which only intersect 
in $x$. Assume that we already constructed some finite number of such paths. Since all of them have 
finite length, there is an index $i$ such that $A_i \sm B_i$ doesn't contain any vertex in their 
union. The ray $R$ has a tail contained in $A_i \sm B_i$ and since $x \in A_i \cap B_i$ we know that 
$x$ has a neighbour in $A_i \sm B_i$. Finally $A_i \sm B_i$ is connected, so we can find a path 
connecting $x$ to the tail of $R$ which intersects the previously constructed paths only in $x$. 
Proceeding inductively we obtain infinitely many paths connecting $x$ to $R$ which pairwise only 
intersect in $x$ completing the proof of the Lemma.
\end{proof}

We would now like to construct a subset of the set $\Srel$ of $\omega$-relevant separations that is 
both nested and invariant under all 
automorphisms and from that set we construct a tree.
The following two lemmas give us important properties of nestedness when we restrict to 
$\omega$-relevant separations.

\begin{lem}\label{options_lem}
Two separations  $(A,B), (C,D)$ in ${\cal S}_\omega$ are nested if and only if they are either 
comparable with respect to $\leq$, or $A\se D$.
\end{lem}

\begin{proof}
First assume that the two separations are nested. It is impossible that $B \se C$ and $D \se A$ 
since the end $\omega$ lies in $B$ and $D$, but not in $C$ and $A$. Hence, if the two separations 
are not comparable, then we know that $A \se D$ and $C \se B$.

For the converse implication first consider the case that $A \se D$.  We want to show that 
$C\subseteq B$.  Assume for a contradiction that there is a vertex $x$ in $C \sm B$. This vertex 
must be contained in $A\se D$ and hence in the separator $C \cap D$. By the definition of $\Srel$ the 
vertex $x$ must have a neighbour $y$ in $C \sm D$. Then $y \notin A$ and $x \notin B$, contradicting 
the fact that the edge $xy$ must lie in either $A$ or $B$, as $(A,B)$ is a separation.

Finally, note that any two separations in ${\cal S}_\omega$ that are comparable with respect to 
$\leq$ are obviously nested.
\end{proof}

\begin{lem}\label{finitely_incomparable}  \rm{(Analogies with \cite[Lemma~4.2]{zbMATH06486846})}
For each $(A,B)\in {\cal S}_\omega$ there are only finitely many $(C,D)\in {\cal S}_\omega$ not 
nested with $(A,B)$. 
\end{lem}

\begin{proof}  The first step is to show that if $(C,D)$ is not nested with $(A,B)$ then $(C,D)$ 
separates some vertices $v$ and $w$ in $A\cap B$.  Then we show that we may assume that the 
separation is minimal.  Since $A\cap B$ is finite there are only finitely many possibilities for the 
pair $v, w$ and we can apply Lemma~\ref{halin} to deduce the result.

First suppose for a contradiction that $(C\sm D)\cap (A\cap B)$ is empty. Since $C\sm D$ is 
connected, it must be a subset of $A\sm B$ or $B\sm A$.
As every vertex in $C \cap D$ has a neighbour in $C \sm D$ it follows that $C\se A$ in the first 
case, whilst $C\se B$ in the second. In both cases $(A,B)$ and $(C,D)$ are nested by 
Lemma~\ref{options_lem}, contrary to our assumption.
Hence there exists a vertex $v\in (C\sm D)\cap (A\cap B)$.   Note that by letting the separations 
$(A,B)$ and $(C,D)$ switch roles we see that $(A\sm B)\cap (C\cap D)$ is also non-empty.

Since the separation $(C,D)$ is in ${\cal S}_\omega$ there is by Lemma~\ref{Lrelevant} a family of 
$k$ disjoint rays that all have their initial vertices in $C\cap D$.  Because $\omega$ lies in $D$, 
all vertices in these rays, except their initial vertices, are contained in the component of 
$D\setminus C$ that contains $\omega$.  Pick a vertex $v'$ from $(A\setminus B)\cap(C\cap D)$.  This 
vertex $v'$ is the initial vertex of one of the rays mentioned above.  Since $\omega$ lies in $B$ 
these rays must contain a vertex $w$ from $A\cap B$ and as mentioned above $w$ is contained in the 
component of $D\setminus C$ that contains $\omega$.   Now we have shown that $(C, D)$ separates the 
two vertices $v$ and $w$.  This separation is minimal because $v$ is in $C\sm D$ and $C\sm D$ is 
connected and has $C\cap D$ as it neighbourhood, and $w$ is contained in the component of $G-(C\cap 
D)$ that contains $\omega$ and that component has the whole of $C\cap D$ as its neighbourhood.
\end{proof}

Let $G$ be a one-ended graph whose end $\omega$ is undominated and has finite 
vertex 
degree $k$.  Recall that by Lemma~\ref{infinitechains} there are no infinite decreasing chains in 
$\Srel$---such a chain would define an end $\mu \neq \omega$, contradicting the assumption that 
$G$ has only one end.  In particular, $\Srel$ has minimal elements.   Assign recursively an ordinal 
$\alpha(A,B)$ to each $(A,B)\in \Srel$ by the following method:
if $(A,B)$ is minimal (with respect to $\leq$ in $\Srel$) then set $\alpha(A,B)=0$;
otherwise  define $\alpha(A,B)$ as the smallest ordinal $\beta$ such that  $\alpha(C,D)<\beta$ for 
all separations $(C,D)\in \Srel$ such that $(C,D)<(A,B)$.  For $v\in V(G)$, let $\Srel(v)$ be the 
set of those separations $(A,B)$ in $\Srel$ with $v\in A\cap B$.
Now set
\[
 \alpha(v)=\sup\{\alpha(A,B) \mid (A,B)\in \Srel(v)\}.
\]
If it so happens that $\Srel(v)$ is empty then $\alpha(v)=0$.
For a vertex set $S$, we let $\alpha(S)$ be the supremum over all $\alpha(v)$ with $v\in S$.  Note 
that the functions $\alpha(A,B)$ and $\alpha(v)$ are both invariant under the action of the 
automorphism group of $G$.

\begin{eg}
Below is a construction of a graph where $\alpha$ takes ordinal values that are not natural numbers.
However, it is not difficult to show that for a locally finite connected graph the $\alpha$-values 
are always natural numbers.

We construct a graph $G$ at which $\alpha$ takes values that are not natural 
numbers. Let $P_n=v_0^n, \ldots, v_n^n$ be a path of length $n$.  We obtain $G$ by taking a ray and identifying its starting vertex $r$ with the vertices $v_n^n$ for each $n\geq 0$.  This graph has only one end $\mu$ and its vertex degree is $1$.
For $0\leq k\leq n-1$ the separation $(\{v_0^n,\ldots, v_k^n\}, V(G)\sm\{v_0^n,\ldots, v_{k-1}^n\})$ is $\mu$-relevant and its $\alpha$-value is $k$.
Hence any separation $(A,B)$ with $r$ (and all the attached paths) in $A$ has $\alpha$-value at 
least the ordinal $\omega$.
\end{eg}

\begin{lem}\label{alm_C}
 Let $G$ be a graph with only one end $\omega$.  Assume that $\omega$ is undominated and has vertex 
degree $k$. 
Let $(C,D)$ be in ${\cal S}_\omega$. Then for all but finitely many vertices $v$ in $C$, we have 
$\alpha(v)\leq \alpha(C,D)$. 
\end{lem}

\begin{proof}
By Lemma~\ref{finitely_incomparable}, there are only finitely many separations in $\Srel$ that are not 
nested with $(C,D)$.
Let $C'$ the set of those vertices in $C\setminus D$ that are not in any separator of these finitely 
many separations.
It suffices to show that if $v\in C'$ and $(A,B)$ in $\Srel(v)$ then  $\alpha(A,B)< \alpha(C,D)$.  
Note that the result is trivially true if $\Srel(v)$ is empty.
By the choice of $v$, the separations $(A,B)$ and $(C,D)$ are nested.
Since $v$ is in $(C\sm D)\cap (A\cap B)$, it is not true that $A \se D$ or $B \se D$.
Since the end $\omega$ does not lie in the sides $A$ and $C$, it does not lie in the side $A\cup C$ 
of the separation $(A\cup C,B\cap D)$.
Hence it lies in the side $B\cap D$. In particular $B\cap (D\sm C)$ is nonempty. 
Thus it is not true that $B \se C$.
Looking at the definition of nestedness we see that $A\se C$.   Hence $(A,B)<(C,D)$ and thus 
$\alpha(A,B)<\alpha(C,D)$ and the result follows.
\end{proof}

\begin{lem}
\label{largersep}
Let $G$ be a graph with only one end $\omega$.  Assume that $\omega$ is undominated and has vertex 
degree $k$. 
For every separation $(C,D)$ in ${\cal S}_\omega$, there is a separation $(A,B) \in {\cal S}_\omega$ 
such that 
$C \se A$ and $\alpha (C,D)<\alpha(A,B) $. 
\end{lem}

\begin{proof}   Let $\{ (A_n, B_n)\}_{n\geq 0}$ be a sequence of $\omega$-relevant separations as 
described in Theorem~\ref{Tsequence}.  
Find a separation $(A, B)$ in this sequence 
such that $C\cap D\subseteq A\sm B$.  Suppose for a contradiction that $C\sm D$ contains a vertex 
$x$ from $A\cap B$.  There is a ray $R$ that has $x$ as a starting vertex and every other vertex is 
contained in $B\sm A$.  Because $C\cap D$ contains no vertex from $B$ we see that this ray would be 
contained in $C\sm D$, contradicting the assumption that the end $\omega$ lies in $D$.  Hence, $C\sm 
D$ does not intersect $A\cap B$ and then, since $C\sm D$ is connected, we conclude that $C\subseteq 
A$.
Thus $ \alpha(C,D)\leq \alpha(A,B)$.

By the previous Lemma there are at most finitely many vertices $v$ in $C$ such that 
$\alpha(v)>\alpha(C,D)$.  Suppose for a contradiction that $v$ is such a vertex and there is no 
value of $n$ such that $\alpha(v)< \alpha(A_n, B_n)$.  Then we can find a sequence $\{(C_n, 
D_n)\}_{n\geq 0}$ of separations in $\Srel(v)$ such that $\alpha(C_1, D_1)<\alpha(C_2, D_2)<\cdots$ 
and for every $n$ there is a number $r_n$ such $\alpha(A_n, B_n)<\alpha(C_{n_r}, B_{n_r})$.   By 
Lemma~\ref{finitely_incomparable} we may assume that for all values of $n$ and $m$ the separations 
$(C_n, D_n)$ and $(C_m, B_m)$ are nested.  Say that a pair of separations $\{(C_n, D_n), (C_m, 
D_m)\}$ is blue if the separations are comparable with respect to $\leq$ and red otherwise.  By 
Ramsey's Theorem, see e.g.\ \cite[(1.9) on p.~16]{zbMATH00044603}, there is an infinite set of 
separations such that all pairs from that set have the same colour.  If all pairs from that set were 
blue then we could find an infinite increasing or a decreasing chain.   By 
Lemma~\ref{infinitechains}(2) there cannot be an infinite descending chain of separations and if there 
was an infinite increasing chain in $\Srel(v)$ then, by Lemma~\ref{infinitechains}(3) with the roles of the $A_i$'s and the $B_i$'s reversed,  $v$ would be a dominating vertex for the end 
$\omega$, contrary to assumptions.  Hence all pairs from that infinite set must be red and we can 
conclude that there is an infinite set of separations in the family $\{(C_n, D_n)\}_{n\geq 0}$ such 
that no two of them are comparable with respect to ordering.  We may assume that if $n$ and $m$ are 
distinct then $(C_n, D_n)$ and $(C_m, D_m)$ are not comparable and then $C_n\sm D_n$ and $C_m\sm 
D_m$ are disjoint.  Start by choosing $n$ such that $v\in A_n\sm B_n$ and then choose $m$ such that 
none of the vertices in $A_n\cap B_n$ is in $C_m\sm D_m$.  There must be some vertex $u$ that 
belongs both to $B_n$ and $C_m\sm D_m$.  The set $(C_m\sm D_m)\cup\{v\}$ is connected and thus it 
contains a $v-u$ path $P$.   But $v\in A_n\sm B_n$ and $u\in B_n\sm A_n$ and the path $P$ contains 
no vertices from $A_n\cap B_n$. We have reached a contradiction.  Hence our original assumption 
must be wrong.  
\end{proof}

Let $X$ be a connected set of vertices which cannot be separated from the end $\omega$ by a 
separation of order less than $k$. A separation $(A,B)\in\Srel$ is called \emph{$X$-nice}, if for 
every $v \in A 
\cap B$ we have $\alpha(v) > \alpha(X)$ and there is some $\varphi \in \Aut(G)$ such that $\varphi(X) 
\se A$ (then we must have $\varphi(X)\subseteq A\sm B$). Let ${\cal N}(X)$ be the set of all 
$X$-nice separations in $\Srel$ which are minimal with respect to $\leq$, i.e.\ $\mathcal N (X)$ 
contains all $X$-nice separations $(A,B) \in \Srel$  such that $A$ is minimal with respect to 
inclusion.

\begin{lem}\label{even_nicer3}
Let $G$ be a graph with only one end $\omega$.  Assume that $\omega$ is undominated and has vertex 
degree $k$.

Suppose $(X,Y)\in \Srel$. 
Then ${\cal N}(X)$ is non-empty.  For each automorphism $\varphi$ of $G$ there is a unique element $(A,B)$ in ${\cal N}(X)$ such that $\varphi(X)\subseteq A$. If $(A,B)$ and $(C,D)$ are not equal and in ${\cal N}(X)$, then $A 
\se D$ and $C\se B$. Furthermore, any two elements of ${\cal N}(X)$ can be mapped onto each other by 
an automorphism.
\end{lem}

\begin{proof}
The existence of an $X$-nice separation follows from Lemma~\ref{largersep}. 
Minimal such separations exist because by Lemma~\ref{infinitechains} an infinite descending chain 
would imply that $G$ had another end $\mu\neq\omega$.

Let $(A,B)$ and $(C,D)$ be elements of ${\cal N}(X)$.   Suppose $\varphi(X) \se A$ and $\psi(X) \se 
C$, where $\varphi, \psi\in\Aut(G)$. Note that $\varphi(X)$ is disjoint from $C \cap D$ because 
$\alpha(\varphi(X)) = \alpha(X)$, which is strictly less than $\alpha(v)$ for any $v\in C\cap D$. 
Hence it is a subset of either $C \sm D$ or $D \sm C$. 
We next prove that if $(A,B)$ and $(C,D)$ are not equal, then $A \se D$ and $C\se B$.

First we consider the case that $\varphi(X)$ is a subset of  $C \sm D$. 
Our aim is to show that $(A,B)$ and $(C,D)$ are equal.  This also implies that $(A,B)$ is the unique element in ${\cal N}(X)$ such that $\varphi(X)\subseteq A$. 
Our strategy will be to construct a $X$-nice separation that is $\leq$ to both of them and by 
minimality of $(A,B)$ and $(C,D)$ we will conclude that it must be equal to
both of them. 
Note that $\varphi(X)$ is included in $(C \sm D)\cap (A\sm B)$. 
Let $A'$ be the connected component of $(C \sm D)\cap (A\sm B)$ that contains  the connected set 
$\varphi(X)$ together with the separator of $(A\cap C, B \cup D)$. 
Let $B'$ be the union of $B \cup D$ with the other components of $(C \sm D)\cap (A\sm B)$.  

Next we show that the separation $(A',B')$ is in $\Ncal(X)$. 
Since the end $\omega$ lies in $B\cap D$, this vertex set is infinite. 
Because $(A,B)$ is in $\Srel$, the separation $(A\cup C, B \cap D)$ has order at least $k$.
Hence by Lemma~\ref{mod}, the separation $(A\cap C, B \cup D)$ has order at most $k$. 
The property that $X$ cannot be separated from $\omega$ by fewer than $k$ vertices implies that the 
separation $(A',B')$ has order precisely $k$. 
Also, every vertex of the separator of $(A',B')$ has a neighbour in $A'\sm B'$ and in $B'\sm A'$.  
Clearly $\omega$ lies in $B'$ and there is no separation $(C',D')$ of order less than $k$ such that 
$A'\se C'$ and $\omega$ lies in $D'$ as $(X,Y)\in \Srel$.
Hence $(A',B')$ is in $\Srel$ and thus it is in $\Ncal(X)$ as $A'\se A$. 
Since $A'\se A$, it must be that $A'=A$ by the minimality of $(A,B)$. Similarly, $A'=C$. Thus $A=C$ 
and so $(A,B)=(C,D)$.
This completes the case when $\varphi(X)$ is a subset of  $C \sm D$.

So we may assume that $\varphi(X) \se D \sm C$, and by symmetry that $\psi(X) \se B \sm A$. 
Consider the separations  $(A \cap D, B \cup C)$ and $(B \cap C, A \cup D)$. 
They must have order at least $k$ because $\varphi(X)\subseteq A\cap D$, $\omega\in B\cup C$ and 
$\psi(X)\subseteq B\cap C$, $\omega\in A\cup D$. So they must have order precisely $k$ by 
Lemma~\ref{mod}.
Let $A'$ be the component of $G-( B \cup C)$  that contains $\varphi(X)$ together with the separator 
of $(A\cap D, B \cup C)$. 
Let $B'$ be the union of $B \cup C$ with the other components. 
Similar as in the last case we show that $(A',B')$ is in $\Ncal(X)$. By the minimality of $(A,B)$ it 
must be that $A \se D$.
The above argument with the separation $(B \cap C, A \cup D)$ in place of  $(A \cap D, B \cup C)$ 
yields that $C \se B$. 
This completes the proof that if $(A,B)$ and $(C,D)$ are not equal and in ${\cal N}(X)$, then $(A,B)$ and $(C,D)$ are nested.

By the above there is for each $\varphi \in \Aut(G)$ a unique separation $(A_\varphi,B_\varphi) \in 
{\cal N}(X)$ such that $\varphi(X) \se A_\varphi$. 
If we apply $\varphi^{-1}$ to this separation we must obtain the unique separation $(A,B) \in {\cal 
N}(X)$  such that $X \se A$. 
Hence any separation of ${\cal N}(X)$ can be mapped by an automorphism to every other separation in 
${\cal N}(X)$.
\end{proof}

\begin{thm}\label{thm:aut_G_inv_td}
Let $G$ be a connected graph with only one end $\omega$, which  is undominated and has finite vertex 
degree $k$. 
Then there is a nested set $\cal S$ of $\omega$-relevant separations of $G$ that is $\Aut(G)$-invariant.  And there is a 1-ended tree $T$ and a bijection between the edge set of $T$ and $\cal 
S$ such that the natural action of $\Aut(G)$ on $S$ induces an action on $T$ by automorphisms.
\end{thm}

\begin{proof}   Pick some $\omega$-relevant separation $(A_0, B_0)$.
Define a sequence $(A_n, B_n)$ of separations as follows. 
For $n \in \mathbb N_{>0}$ pick $(A_n, B_n) \in \mathcal N(A_{n-1})$ such that $A_{n-1} \subsetneq 
A_n$, 
which is possible by Lemma~\ref{even_nicer3}.  Observe that the sequence of separations $(A_n, B_n)$ 
has the same properties as the sequence in Theorem~\ref{Tsequence}.

Now let 
\[
\mathcal S = \{(\varphi(A_n), \varphi(B_n))\mid n \in \mathbb N_{>0}, \varphi \in \Aut(G)\}.
\]
Note that $(A_0,B_0)$ is not an element in  $\Scal$. 

First we prove that $\Scal$ is nested. 
Let $(\varphi(A_n), \varphi(B_n))$ and $(\psi (A_m), \psi(B_m))$ be two different elements of 
$\mathcal S$ (here $\varphi$ and $\psi$ are automorphisms of $G$). 
If $m = n$ then they are nested by Lemma~\ref{even_nicer3}, since they both are elements of $\mathcal 
N(A_{n-1})$. 
Hence assume without loss of generality that $n < m$.  If $\varphi(A_m) = \psi(A_m)$ then 
$\varphi(A_n)\se\varphi(A_m)= \psi(A_{m})$ which implies that the two separations are nested. Otherwise by 
Lemma~\ref{even_nicer3} we have $\varphi(A_n) \se \varphi(A_{m}) \se \psi(B_m)$, also showing 
nestedness, by Lemma~\ref{options_lem}. 

Next we construct a directed graph $T_+$.
We define $T_+$ as follows. Its vertex set is $\Scal$. We add a directed edge from $(\varphi(A_n), 
\varphi(B_n))$ to $(\psi(A_{n+1}), \psi(B_{n+1}))$
if $\varphi(A_n)$ is a subset of $\psi(A_{n+1})$. By \autoref{even_nicer3}, each vertex has 
outdegree at most one. And by the construction of $\Scal$
it has outdegree at least one. 

The next step is to show that the graph is connected. Let 
$(C,D)=\varphi(A_n, B_n)$ be a vertex in $T_+$.  Find an $m$ such that $C\subseteq A_m\sm B_m$.  
Suppose for a contradiction that $(\varphi(A_m), \varphi(B_m))\neq (A_m, B_m)$.  Both 
$(\varphi(A_m), \varphi(B_m))$ and 
$(A_m, B_m)$ are in ${\cal N}(X)$. By Lemma~\ref{even_nicer3} $\varphi(A_m)\subseteq B_m$.
Thus $\varphi(A_m)$ is empty. This is a contradiction to the assumption that $(A_m, B_m)$ is a 
proper separation. Now we see that $$(A_m, B_m)=(\varphi(A_m), \varphi(B_m)), (\varphi(A_{m-1}), 
\varphi(B_{m-1})), \ldots, (\varphi(A_n), \varphi(B_n))=(C,D)$$ 
is a path in $T_+$ from $(A_m, B_m)$ to $(C,D)$.  
Thus every vertex in $T_+$ is in the same connected component as some vertex 
$(A_m, B_m)$ and since they all belong to the same component we deduce that $T_+$ is connected.  
Hence the corresponding undirected graph $T$ is a tree.

The map that 
sends $(\varphi(A_n), \varphi(B_n))$ to the edge with endvertices $(\varphi(A_n), \varphi(B_n))$ 
and 
$(\psi(A_{n+1}), \psi(B_{n+1}))$  is clearly a bijection.
If the ray $(A_1,B_1), (A_2, B_2), \ldots$ is removed from $T$ then what remains of $T$ is clearly rayless and thus the 
tree $T$ is one-ended. 

The statement about the action of $\Aut(G)$ on $T$ follows easily since the properties used to 
define 
$T$ are invariant under $\Aut(G)$.
\end{proof}

A \emph{tree-decomposition} of a graph $G$ consists of a tree $T$ and a family $(P_t)_{t \in V(T)}$ 
of subsets of $V(G)$, one for each vertex of $T$ such that 
\begin{enumerate}[label=(T\arabic*)]
\item $V(G) = \bigcup_{t \in V(T)} P_t$, \item for every edge $e \in E(G)$ there is $t \in V(T)$ 
such that both endpoints of $e$ lie in $P_t$, and \item $P_{t_1} \cap P_{t_3} \subseteq P_{t_2}$ 
whenever $t_2$ lies on the unique path connecting $t_1$ and $t_3$ in $T$.
\end{enumerate}
The tree $T$ is called \emph{decomposition tree}, the sets $P_t$ are called the \emph{parts} of the 
tree-decomposition. 

We associate to an edge $e = st$ of the decomposition tree a separation of $G$ as follows. Removing 
$e$ from $T$ yields two components $T_s$ and $T_t$. Let $X_s = \bigcup _{u \in T_s} P_u$ and $X_t = 
\bigcup _{u \in T_t} P_u$. If $X_s \sm X_t$ and $X_t \sm X_s$ are non-empty (this will be the case 
for all tree-decompositions considered in this paper), then $(X_s,X_t)$ is a proper separation of $G$. 
Clearly, the set of all separations associated to edges of a decomposition tree is nested.

The separators $A \cap B$ of the separations associated to edges of a decomposition 
tree are called \emph{adhesion sets}. The supremum of the sizes of adhesion sets is called the 
\emph{adhesion} of the tree-decomposition. The tree-decompositions constructed in this 
paper all have finite adhesion.

Given a graph $G$ with only one end $\omega$ and a tree-decomposition $(T,P_t\mid t\in V(T))$ of 
$G$ of finite adhesion, then  $(T,P_t\mid t\in V(T))$ \emph{displays} $\omega$ if firstly the 
decomposition tree $T$ has only one end; call it $\mu$. And secondly for any edge $st$ of $T$ with 
$\mu$ in 
$T_t$, the associated separation $(X_s,X_t)$ has the property that $\omega$ lies in $X_t$. 

A tree-decomposition is \emph{$\Aut(G)$-invariant} if the set $S$ of separations associated to it is 
closed by the natural action of $\Aut(G)$ on $S$.
The following implies Theorem~\ref{main:intro}. 

\begin{thm}\label{main_td}
Let $G$ be a connected graph with only one end $\omega$, which is undominated and has finite 
vertex 
degree $k$. 
Then $G$ has a tree-decomposition $(T,P_t\mid t\in V(T))$ of adhesion $k$ that displays $\omega$ 
and is $\Aut(G)$-invariant.
\end{thm}

\begin{proof}
 We follow the notation of the proof of Theorem~\ref{thm:aut_G_inv_td}. 
 
 Given a vertex $t$ of $T_+$, the \emph{inward neighbourhood} of $t$, denoted by $N_+(t)$, is  
the set of vertices $u$ of $T_+$ such that there is a directed edge from $u$ to $t$ in $T_+$. 
Recall that the vertices of $T_+$ are (in bijection with) separations; we refer to the 
separation 
associated to the vertex $t$ by $(A_t,B_t)$. Given a vertex $t$, we let $P_t=A_t\sm \bigcup_{u\in 
N_+(t)} (A_u\sm B_u)$. 

It is straightforward that $(T,P_t\mid t\in V(T))$ is a tree-decomposition of adhesion $k$ (whose set of associated separations is $\Scal\cup \{(B,A)\mid (A,B)\in \Scal\}$). 
It is not hard to see that $(T,P_t\mid t\in V(T))$ displays $\omega$ 
and is $\Aut(G)$-invariant. 
\end{proof}

\begin{eg}\label{undom_nec}
In this example we construct a one-ended graph $G$ whose end is dominated and has vertex degree 1, but the graph $G$ has no tree-decomposition of finite adhesion that is invariant under the group of automorphisms and whose decomposition tree is one-ended.
We obtain $G$ from the canopy tree by adding a new vertex adjacent to all the leaves of the canopy tree. Then we add infinitely many vertices of degree one only incident to that new vertex, see Figure~\ref{fig:dom}.

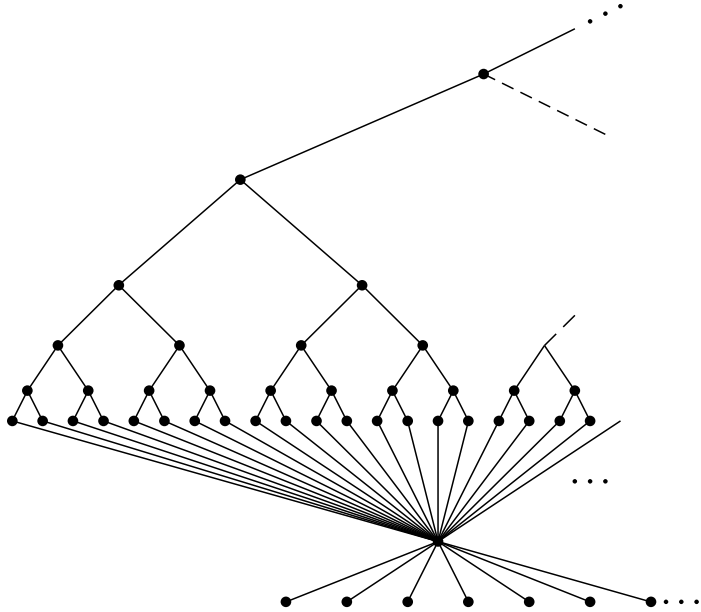
\begin{figure}
\centering
%
%
\psscalebox{1.0 1.0} 
{
\begin{pspicture}(0,-3.9992788)(9.098557,3.9992788)
\psdots[linecolor=black, dotsize=0.14](0.0689904,-1.5302885)
\psdots[linecolor=black, dotsize=0.14](0.46899042,-1.5302885)
\psdots[linecolor=black, dotsize=0.14](0.8689904,-1.5302885)
\psdots[linecolor=black, dotsize=0.14](1.2689904,-1.5302885)
\psdots[linecolor=black, dotsize=0.14](1.6689904,-1.5302885)
\psdots[linecolor=black, dotsize=0.14](2.0689905,-1.5302885)
\psdots[linecolor=black, dotsize=0.14](2.4689903,-1.5302885)
\psdots[linecolor=black, dotsize=0.14](2.8689904,-1.5302885)
\psdots[linecolor=black, dotsize=0.14](0.2689904,-1.1302884)
\psdots[linecolor=black, dotsize=0.14](1.0689903,-1.1302884)
\psdots[linecolor=black, dotsize=0.14](1.8689904,-1.1302884)
\psdots[linecolor=black, dotsize=0.14](2.6689904,-1.1302884)
\psdots[linecolor=black, dotsize=0.14](0.6689904,-0.5302884)
\psdots[linecolor=black, dotsize=0.14](2.2689905,-0.5302884)
\psdots[linecolor=black, dotsize=0.14](1.4689904,0.26971158)
\psline[linecolor=black, linewidth=0.02](0.0689904,-1.5302885)(0.2689904,-1.1302884)
\psline[linecolor=black, linewidth=0.02](0.46899042,-1.5302885)(0.2689904,-1.1302884)
\psline[linecolor=black, linewidth=0.02](0.8689904,-1.5302885)(1.0689903,-1.1302884)
\psline[linecolor=black, linewidth=0.02](1.2689904,-1.5302885)(1.0689903,-1.1302884)
\psline[linecolor=black, linewidth=0.02](1.0689903,-1.1302884)(0.6689904,-0.5302884)
\psline[linecolor=black, linewidth=0.02](0.6689904,-0.5302884)(0.2689904,-1.1302884)
\psline[linecolor=black, linewidth=0.02](0.6689904,-0.5302884)(1.4689904,0.26971158)
\psline[linecolor=black, linewidth=0.02](1.4689904,0.26971158)(2.2689905,-0.5302884)
\psline[linecolor=black, linewidth=0.02](2.2689905,-0.5302884)(1.8689904,-1.1302884)
\psline[linecolor=black, linewidth=0.02](1.8689904,-1.1302884)(1.6689904,-1.5302885)
\psline[linecolor=black, linewidth=0.02](1.8689904,-1.1302884)(2.0689905,-1.5302885)
\psline[linecolor=black, linewidth=0.02](2.4689903,-1.5302885)(2.6689904,-1.1302884)(2.6689904,-1.1302884)
\psline[linecolor=black, linewidth=0.02](2.6689904,-1.1302884)(2.8689904,-1.5302885)(2.8689904,-1.5302885)
\psline[linecolor=black, linewidth=0.02](2.6689904,-1.1302884)(2.2689905,-0.5302884)
\psdots[linecolor=black, dotsize=0.14](3.2689905,-1.5302885)
\psdots[linecolor=black, dotsize=0.14](3.6689904,-1.5302885)
\psdots[linecolor=black, dotsize=0.14](4.06899,-1.5302885)
\psdots[linecolor=black, dotsize=0.14](4.4689903,-1.5302885)
\psdots[linecolor=black, dotsize=0.14](4.8689904,-1.5302885)
\psdots[linecolor=black, dotsize=0.14](5.2689905,-1.5302885)
\psdots[linecolor=black, dotsize=0.14](5.6689906,-1.5302885)
\psdots[linecolor=black, dotsize=0.14](6.06899,-1.5302885)
\psdots[linecolor=black, dotsize=0.14](3.4689903,-1.1302884)
\psdots[linecolor=black, dotsize=0.14](4.2689905,-1.1302884)
\psdots[linecolor=black, dotsize=0.14](5.06899,-1.1302884)
\psdots[linecolor=black, dotsize=0.14](5.8689904,-1.1302884)
\psdots[linecolor=black, dotsize=0.14](3.8689904,-0.5302884)
\psdots[linecolor=black, dotsize=0.14](5.4689903,-0.5302884)
\psdots[linecolor=black, dotsize=0.14](4.6689906,0.26971158)
\psline[linecolor=black, linewidth=0.02](3.2689905,-1.5302885)(3.4689903,-1.1302884)
\psline[linecolor=black, linewidth=0.02](3.6689904,-1.5302885)(3.4689903,-1.1302884)
\psline[linecolor=black, linewidth=0.02](4.06899,-1.5302885)(4.2689905,-1.1302884)
\psline[linecolor=black, linewidth=0.02](4.4689903,-1.5302885)(4.2689905,-1.1302884)
\psline[linecolor=black, linewidth=0.02](4.2689905,-1.1302884)(3.8689904,-0.5302884)
\psline[linecolor=black, linewidth=0.02](3.8689904,-0.5302884)(3.4689903,-1.1302884)
\psline[linecolor=black, linewidth=0.02](3.8689904,-0.5302884)(4.6689906,0.26971158)
\psline[linecolor=black, linewidth=0.02](4.6689906,0.26971158)(5.4689903,-0.5302884)
\psline[linecolor=black, linewidth=0.02](5.4689903,-0.5302884)(5.06899,-1.1302884)
\psline[linecolor=black, linewidth=0.02](5.06899,-1.1302884)(4.8689904,-1.5302885)
\psline[linecolor=black, linewidth=0.02](5.06899,-1.1302884)(5.2689905,-1.5302885)
\psline[linecolor=black, linewidth=0.02](5.6689906,-1.5302885)(5.8689904,-1.1302884)(5.8689904,-1.1302884)
\psline[linecolor=black, linewidth=0.02](5.8689904,-1.1302884)(6.06899,-1.5302885)(6.06899,-1.5302885)
\psline[linecolor=black, linewidth=0.02](5.8689904,-1.1302884)(5.4689903,-0.5302884)
\psline[linecolor=black, linewidth=0.02](1.4689904,0.26971158)(3.0689905,1.6697116)
\psline[linecolor=black, linewidth=0.02](4.6689906,0.26971158)(3.0689905,1.6697116)
\psdots[linecolor=black, dotsize=0.14](3.0689905,1.6697116)
\psdots[linecolor=black, dotsize=0.14](6.4689903,-1.5302885)
\psdots[linecolor=black, dotsize=0.14](6.8689904,-1.5302885)
\psdots[linecolor=black, dotsize=0.14](7.2689905,-1.5302885)
\psdots[linecolor=black, dotsize=0.14](7.6689906,-1.5302885)
\psline[linecolor=black, linewidth=0.02](6.4689903,-1.5302885)(6.6689906,-1.1302884)
\psline[linecolor=black, linewidth=0.02](6.6689906,-1.1302884)(6.8689904,-1.5302885)
\psline[linecolor=black, linewidth=0.02](7.2689905,-1.5302885)(7.4689903,-1.1302884)
\psline[linecolor=black, linewidth=0.02](7.4689903,-1.1302884)(7.6689906,-1.5302885)
\psdots[linecolor=black, dotsize=0.14](6.6689906,-1.1302884)
\psdots[linecolor=black, dotsize=0.14](7.4689903,-1.1302884)
\psline[linecolor=black, linewidth=0.02](6.6689906,-1.1302884)(7.06899,-0.5302884)
\psline[linecolor=black, linewidth=0.02](7.4689903,-1.1302884)(7.06899,-0.5302884)
\psline[linecolor=black, linewidth=0.02, linestyle=dashed, dash=0.17638889cm 0.10583334cm](7.06899,-0.5302884)(7.4689903,-0.1302884)
\psline[linecolor=black, linewidth=0.02](3.0689905,1.6697116)(6.2689905,3.0697117)
\psline[linecolor=black, linewidth=0.02, linestyle=dashed, dash=0.17638889cm 0.10583334cm](6.2689905,3.0697117)(7.8689904,2.2697115)
\psline[linecolor=black, linewidth=0.02](6.2689905,3.0697117)(7.4689903,3.6697116)
\psdots[linecolor=black, dotsize=0.14](5.6689906,-3.1302884)
\psline[linecolor=black, linewidth=0.02](5.6689906,-3.1302884)(0.0689904,-1.5302885)
\psline[linecolor=black, linewidth=0.02](0.46899042,-1.5302885)(5.6689906,-3.1302884)
\psline[linecolor=black, linewidth=0.02](5.6689906,-3.1302884)(0.8689904,-1.5302885)
\psline[linecolor=black, linewidth=0.02](1.2689904,-1.5302885)(5.6689906,-3.1302884)
\psline[linecolor=black, linewidth=0.02](5.6689906,-3.1302884)(2.0689905,-1.5302885)
\psline[linecolor=black, linewidth=0.02](5.6689906,-3.1302884)(1.6689904,-1.5302885)
\psline[linecolor=black, linewidth=0.02](5.6689906,-3.1302884)(2.4689903,-1.5302885)
\psline[linecolor=black, linewidth=0.02](2.8689904,-1.5302885)(5.6689906,-3.1302884)
\psline[linecolor=black, linewidth=0.02](5.6689906,-3.1302884)(3.2689905,-1.5302885)
\psline[linecolor=black, linewidth=0.02](3.6689904,-1.5302885)(5.6689906,-3.1302884)
\psline[linecolor=black, linewidth=0.02](5.6689906,-3.1302884)(4.06899,-1.5302885)
\psline[linecolor=black, linewidth=0.02](5.6689906,-3.1302884)(4.4689903,-1.5302885)
\psline[linecolor=black, linewidth=0.02](5.6689906,-3.1302884)(4.8689904,-1.5302885)
\psline[linecolor=black, linewidth=0.02](5.6689906,-3.1302884)(5.2689905,-1.5302885)
\psline[linecolor=black, linewidth=0.02](5.6689906,-3.1302884)(5.6689906,-1.5302885)
\psline[linecolor=black, linewidth=0.02](5.6689906,-3.1302884)(6.06899,-1.5302885)
\psline[linecolor=black, linewidth=0.02](5.6689906,-3.1302884)(6.4689903,-1.5302885)
\psline[linecolor=black, linewidth=0.02](5.6689906,-3.1302884)(6.8689904,-1.5302885)
\psline[linecolor=black, linewidth=0.02](5.6689906,-3.1302884)(7.2689905,-1.5302885)
\psline[linecolor=black, linewidth=0.02](5.6689906,-3.1302884)(7.6689906,-1.5302885)
\psline[linecolor=black, linewidth=0.02](5.6689906,-3.1302884)(8.068991,-1.5302885)
\psdots[linecolor=black, dotsize=0.14](6.2689905,3.0697117)
\psdots[linecolor=black, dotsize=0.06](7.4689903,-2.3302884)
\psdots[linecolor=black, dotsize=0.06](7.6689906,-2.3302884)
\psdots[linecolor=black, dotsize=0.06](7.8689904,-2.3302884)
\psdots[linecolor=black, dotsize=0.06](7.6689906,3.7697115)
\psdots[linecolor=black, dotsize=0.06](7.8689904,3.8697116)
\psdots[linecolor=black, dotsize=0.06](8.068991,3.9697115)
\psdots[linecolor=black, dotsize=0.14](3.6689904,-3.9302883)
\psdots[linecolor=black, dotsize=0.14](4.4689903,-3.9302883)
\psdots[linecolor=black, dotsize=0.14](5.2689905,-3.9302883)
\psdots[linecolor=black, dotsize=0.14](6.06899,-3.9302883)
\psdots[linecolor=black, dotsize=0.14](6.8689904,-3.9302883)
\psdots[linecolor=black, dotsize=0.14](7.6689906,-3.9302883)
\psline[linecolor=black, linewidth=0.02](5.6689906,-3.1302884)(3.6689904,-3.9302883)
\psline[linecolor=black, linewidth=0.02](4.4689903,-3.9302883)(5.6689906,-3.1302884)
\psline[linecolor=black, linewidth=0.02](5.6689906,-3.1302884)(5.2689905,-3.9302883)
\psline[linecolor=black, linewidth=0.02](5.6689906,-3.1302884)(6.06899,-3.9302883)
\psline[linecolor=black, linewidth=0.02](5.6689906,-3.1302884)(6.8689904,-3.9302883)
\psline[linecolor=black, linewidth=0.02](5.6689906,-3.1302884)(7.6689906,-3.9302883)
\psline[linecolor=black, linewidth=0.02](5.6689906,-3.1302884)(8.46899,-3.9302883)
\psdots[linecolor=black, dotsize=0.14](8.46899,-3.9302883)
\psdots[linecolor=black, dotsize=0.06](8.66899,-3.9302883)
\psdots[linecolor=black, dotsize=0.06](8.86899,-3.9302883)
\psdots[linecolor=black, dotsize=0.06](9.068991,-3.9302883)
\end{pspicture}
}
\caption{A graph with no $\Aut(G)$-invariant tree-decomposition of finite adhesion.}
\label{fig:dom}
\end{figure}

Suppose for a contradiction that $G$ has a tree-decomposition $(T,P_t\mid t\in V(T))$ of finite adhesion that is invariant under the group of automorphisms and such that $T$ is one-ended.

There cannot be a single part $P_t$ that contains a ray of the canopy tree. To see that first note that there cannot be two such parts by the assumption of finite adhesion. Hence any such part would contain all vertices of the canopy tree from a certain level onwards. This is not possible by finite adhesion. 

Having shown that there cannot be a single part $P_t$ that contains a ray of the canopy tree, it must be that every part $P_t$ with $t$ near enough to the end of $T$ contains a vertex of the canopy tree. 

Our aim is to show that any vertex $u$ of degree 1 is in all parts. Suppose not for a contradiction. Then since $T$ is one-ended, there is a vertex $t$ of $T$ such that $t$ separates in $T$ all vertices $s$ with $u\in P_s$ from the end of $T$. We pick $t$ high enough in $T$ such that there is a vertex $v$ of the canopy tree in $P_t$. If $P_t$ contained all vertices of the orbit of $v$, then $P_t$ together with all parts $P_s$, where $s$ has some fixed bounded distance from $t$ in $T$, would contain a ray. This is impossible; the proof is similar as that that $P_t$ cannot contain a ray. Hence there is a vertex $v'$ in the orbit of $v$ that is not in $P_t$. Take an automorphism of $G$ that fixes $u$ and moves $v$ to $v'$. As the tree-decomposition is $\Aut(G)$-invariant, $T$ has a vertex $s$ such that $u,v'\in P_s$ but $v\notin P_s$. Since $T$ is $\Aut(G)$-invariant and one-ended, $t$ does not separate $s$ from the end of $T$. This is a contradiction as $u\in P_s$.

Hence $u$ must be in all parts. As $u$ was arbitrary, every vertex of degree one must be in every part. So the tree-decomposition does not have finite adhesion. This is the desired contradiction. Hence such a tree-decomposition does not exist.

\end{eg}

\section{A dichotomy result for automorphism groups}
\label{sec:dichotomy}

Before we turn to a proof of Theorem~\ref{dichotomy}, we state a few helpful auxiliary results. The 
following lemma can be seen as a consequence of \cite[Lemma~7]{zbMATH01701677}, but for 
completeness 
a direct proof is provided in Appendix B.

\begin{lem}
\label{lem:treefixtail}
If $T$ is a one-ended tree and $R$ is a ray in $T$, then every automorphism of $T$ fixes some tail 
of $R$ pointwise.
\end{lem}

The next result is Lemma 3 in \cite{zbMATH01701677}.   For completeness a proof is included in 
Appendix C.

\begin{lem}
\label{Lrayless}
The pointwise (and hence also the setwise) stabiliser of a finite set of vertices in the 
automorphism group of a rayless graph is either finite or contains at least $2^{\aleph_0}$ many 
elements.
\end{lem}

The next result is an extension of Lemma~\ref{Lrayless} to one-ended graphs where the end has finite 
vertex degree.

\begin{lem}
\label{lem:dichotomy}
Let $G$ be a graph with only one end $\omega$.  Assume that $\omega$ has finite vertex degree $k$.  Let $X$ be a finite set of 
vertices in $G$ that contains all the vertices that dominate the end.  If the graph $G-X$ is 
connected then the pointwise stabiliser of $X$ in $\Aut(G)$ is either finite or contains at least 
$2^{\aleph_0}$ many elements.
\end{lem}

\begin{proof}
Denote by $\Gamma$ the pointwise stabiliser of $X$ in $\Aut(G)$. If $\Gamma$ is finite, then there 
is 
nothing to show, hence assume that $\Gamma$ is infinite.

Consider a nested $ \Aut (G-X)$-invariant set of $\omega$-relevant separations of $G-X$ as in 
Theorem~\ref{thm:aut_G_inv_td} and a  tree $T$ built from this set in the way described. Clearly 
$\Gamma$ gives rise to a subgroup of $\Aut (G - X)$ whence this nested set is $\Gamma$-invariant. 
Adding $X$ to both sides of every separation in $\cal S$ gives rise to a new $\Gamma$ invariant set 
$\cal S$ of nested separations such that each separation has order $k+|X|$.  The tree we get from 
$\cal S$ is the same as $T$.  From now on we will work with $\cal S$.

Every element $\gamma \in \Gamma$ induces an automorphism of $T$. Note that this canonical 
action of $\Gamma$ on $T$ is in general not faithful, i.e.\ it is possible that different elements 
of $\Gamma$ induce the same automorphism of $T$.

Let $R$ be a ray in $T$ and let $(e_n)_{n \in \mathbb N}$ be the family of edges of $R$ (in the 
order in which they appear on $R$). Let $(A_n,B_n)$ be the separation of $G$ corresponding to 
$e_n$. 
Denote by $\Gamma_n$ the stabiliser of $e_n$ in $\Gamma$. By Lemma~\ref{lem:treefixtail} every 
automorphism of $T$ (and hence also every $\gamma \in \Gamma$) fixes some tail of $R$, so 
$\Gamma_n$ 
is non-trivial for large enough $n$. Furthermore, $\Gamma_n$ is a subgroup of $\Gamma_m$ whenever 
$n 
\leq m$. 

We claim that for all but finitely many $n$, we have at least one non-trivial $\gamma$ in the 
pointwise stabiliser of $B_n$. To see this, let $\gamma_1, \ldots, \gamma_{(k+|X|)!+1}$ be a set of 
$(k+|X|)!+1$ different non-trivial automorphisms in $\Gamma$. Choose $n$ large enough such that they 
all are contained in $\Gamma_n$ and act differently on $A_n$. By a simple pigeon hole argument, at 
least two of them, $\gamma_1$ and $\gamma_2$ say, have the same action on $A_n \cap B_n$. Then 
$\gamma_1 \circ \gamma_2^{-1}$ is an automorphism which fixes $A_n \cap B_n$ pointwise, and fixes 
$A_n$ setwise but not pointwise. Now, using the \emph{independence property} from 
Section~\ref{sec:autogroup} we can define an automorphism 
\[
	\gamma(x) =
	\begin{cases}
	\gamma_1 \circ \gamma_2^{-1}(x) & \text{if }x \in A_n\sm B_n\\
	x & \text{if }x \in B_n
	\end{cases}
\]
with the desired properties.

Note that the subgroup leaving $A_n$ invariant in the pointwise stabiliser of $B_n$ in $\Gamma$ 
induces the same permutation group on  the rayless graph induced by $A_n$ in $G$ as does the 
subgroup leaving $A_n$ invariant in the pointwise stabiliser of $A_n\cap B_n$. 
Hence, if there is $n \in \mathbb N$ such that the pointwise stabiliser of $B_n$ in $\Gamma$ is 
infinite, then this stabiliser contains at least $2^{\aleph_0}$ many elements by Lemma~\ref{Lrayless}.

So (by passing to a tail of $R$) we may assume that the pointwise stabiliser of $B_n$ is a finite 
but non-trivial subgroup of $\Gamma$ for every $n \in N$.

Next we claim that for every $n$ there is a non-trivial automorphism in the pointwise stabiliser of 
$A_n $. If not, then $\Gamma_n$ is finite and we choose $\sigma \in \Gamma \sm \Gamma_n$. For an 
edge $e$ of $T$, denote by $T_e$ the component of $T-e$ which does not contain the end of $T$.  
Clearly $\sigma(T_e) = T_{\sigma (e)}$ for every edge $e$. In particular, if $e = e_m$ is the last 
edge of $R$ which is not fixed by $\sigma$, then clearly $\sigma (T_e) \se T - T_{e}$. Furthermore 
$n < m$,  so $A_n \se A_m$, and $B_m \se B_n$. Hence $\sigma(A_n) 
\se \sigma(A_m) 
\se  B_m
\se B_n$.
Now let $\gamma$ be a nontrivial automorphism in the pointwise stabiliser of $B_n$. Then 
$\sigma^{-1} \circ \gamma \circ \sigma$ is easily seen to be a nontrivial element of the pointwise 
stabiliser of $A_n$: for $a \in A_n$ we have
\[
\sigma^{-1} \circ \gamma \circ \sigma (a) = \sigma^{-1} \circ \sigma (a) = a
\]
since $\sigma(a) \in B_n$ is fixed by $\gamma$.

Now define an infinite sequence $(\gamma_k)_{k \in \mathbb N}$ of elements of $\Gamma$ as follows. 
Pick a nontrivial $\gamma_1$ in the pointwise stabiliser of $A_1$. Assume that $\gamma_i$ has been 
defined for $i < k$, then let $n_k$ be such that $\gamma_i$ acts non-trivially on $A_{n_k}$ for all 
$i<k$ and pick a nontrivial element $\gamma_k$ in the pointwise stabiliser of $A_{n_k}$. For an 
infinite $0$-$1$-sequence $(r_j)_{j\geq 1}$, define
\[
	\psi_i = \gamma_i^{r_i} \circ \gamma_{i-1}^{r_{i-1}} \circ \cdots \circ \gamma_1^{r_1},
\]
in other words, $\psi_n$ is the composition of all $\gamma_j$ with $j\leq n$ and $r_j = 1$. Finally 
define $\psi$ to be the limit of the $\psi_n$ in the topology of pointwise convergence. This limit 
exists, because for $j > i$ the restriction $\psi_i$ and $\psi_j$ to $A_{n_i}$ coincide, and the 
$A_{n_i}$ exhaust $V(G)$. By Lemma~\ref{aut-closed}, $\psi$ is contained in $\Aut(G)$  and is also in 
$\Gamma \se \Aut(G)$ because every $\psi_i$  stabilises $X$ pointwise. 

Finally assume that we have two different $0$-$1$-sequences $(r_j)_{j\geq 1}$ and $(r'_j)_{j\geq 
1}$ 
and let $(\psi_j)_{j\geq 1}$ and $(\psi'_j)_{j\geq 1}$ be the corresponding sequences of 
automorphisms. If $l$ is the first index such that $r_{l} \neq r_{l}'$ then the restrictions of  
$\psi_{l}$ and $\psi_{l}'$ (and hence also of $\psi_i$ and $\psi_i'$ for $i>l$) to $A_{n_l}$ 
differ. 
Hence different $0$-$1$-sequences give different elements of $\Gamma$ and $\Gamma$ contains at 
least 
$2^{\aleph_0}$ many elements.
\end{proof}

\begingroup
\def\thethm{\ref{dichotomy}}
\begin{thm}
Let $G$ be a graph with one end which has finite vertex degree. Then $\Aut(G)$ is either finite or 
has at least $2^{\aleph_0}$ many elements.
\end{thm}
\addtocounter{thm}{-1}
\endgroup

\begin{proof}
Let $X$ be the set of vertices which dominate $\omega$.  This set is possibly empty and by 
Lemma~\ref{lem:finite_dominating} it is finite.  Every automorphism stabilises $X$ setwise.  
Therefore the pointwise stabiliser of $X$ is a normal subgroup of $\Aut(G)$ with finite index. So it 
suffices to show that the conclusion of Theorem~\ref{dichotomy} holds for the stabiliser $\Gamma$ of 
$X$.

For every component $C$ of $G-X$ let $\Gamma_C$ be the pointwise stabiliser of $X$ in $\Aut(C \cup 
X)$. Then $\Gamma_C$ is either finite or contains at least $2^{\aleph_0}$ many elements by 
Lemma~\ref{Lrayless} and Lemma~\ref{lem:dichotomy}. If $|\Gamma_C|=2^{\aleph_0}$ for some component $C$ 
then we need do no more.  So assume that all the groups $\Gamma_C$ are finite. The same argument as 
used towards the end of the proof of Lemma~\ref{Lrayless} (see Appendix C) now shows that either $\Gamma$ is finite 
or has at least cardinality $2^{\aleph_0}$.
\end{proof}

As a corollary we can answer a question posed by Boutin and Imrich in \cite{imrichboutin}.
In order to state this question, we first need some notation.
For a vertex $v$ in a graph $G$ we define $B_v(n)$, \emph{the ball of radius $n$ centered at $v$},  
as the set of all vertices in $G$ in distance at most $n$ from $v$.  We also define $S_v(n)$, 
\emph{the sphere of radius $n$ centered at $v$},  as the set of all vertices in $G$ in distance 
exactly $n$ from $v$.  A connected locally finite graph is said to have \emph{linear growth} if 
there is a constant $c$ such that $|B_v(n)|\leq cn$ for all $n= 1, 2, \ldots$.   It is an easy 
exercise to show that the property of having linear growth does not depend on the choice of the 
vertex $v$.  

In relation to their work on the distinguishing cost of graphs Boutin and Imrich \cite{imrichboutin} 
ask whether there exist one-ended locally finite graphs that has linear 
growth and countably infinite 
automorphism group.  

If $G$  is a locally finite graph with linear growth and $v$ is a vertex in $G$ then there is a 
constant $k$ such that $|S_v(n)|=k$ for infinitely many values of $n$.  (This is observed by Boutin 
and Imrich in their paper \cite[Fact 2 in the proof of Proposition 13]{imrichboutin}.)  From this 
we  deduce that the vertex-degree of an end of $G$ is at most equal to $k$, since each ray in $G$ 
must pass through all but finitely many of the spheres $S_v(n)$.  Using Theorem~\ref{dichotomy} one can 
now give a negative answer to the above question.

\begin{thm}\label{thm14}
If $G$ is a connected locally finite graph with one end and linear growth, then the automorphism group of $G$ is 
either finite or contains exactly $2^{\aleph_0}$ many elements.
\end{thm}

\begin{proof}
Since $G$ is locally finite and connected, the graph $G$ is countable.   Hence the 
automorphism group cannot contain more than $2^{\aleph_0}$ many elements.  Furthermore linear 
growth 
implies that all ends must have finite vertex degree, hence we can apply Theorem~\ref{dichotomy}.
\end{proof}

In particular a connected graph with linear growth and a countably infinite autormorphism group 
cannot have one end.
  Thus one can strengthen \cite[Theorem 22]{imrichboutin} and get:

\begin{thm}  {\rm (Cf.~\cite[Theorem 22]{imrichboutin})} Every locally finite connected graph with 
linear growth and countably infinite automorphism group has 2 ends.
\end{thm}

Furthermore one can in \cite[Theorem 18]{imrichboutin} remove the assumption that the graph is 
2-ended, since it is implied by the other assumptions.

\section{Ends of quasi-transitive graphs}
\label{sec:qt-ends}

Finally, another application was 
pointed out to the authors by Matthias Hamann. Recall that a graph is called \emph{transitive}, if all vertices lie in the same orbit  under the automorphism group, and \emph{quasi-transitive} (or \emph{almost-transitive}), if there are 
only finitely many orbits on the vertices. 

The groundwork for the study of automorphisms of infinite graphs was laid in the 1973 paper of Halin \cite{zbMATH03417497}.  Among the results there is a classification of automorphisms of a connected infinite graph, see \cite[Sections 5, 6 and 7]{zbMATH03417497}.   \emph{Type 1} automorphisms, to use Halin's terminology, leave a finite set of vertices invariant.  An automorphism is said to be of \emph{type 2} if it is not of type 1.  Type 2 automorphism are of two kinds, the first kind fixes precisely one end which is then thick (i.e.\ has infinite vertex degree) and the second kind fixes precisely two ends which are then both thin (i.e.\ have finite vertex degrees).  In Halin's paper these results are stated with the additional assumption that the graph is locally finite but the classification remains true without this assumption.  

It is a well known fact that a connected, transitive graph has either $1$, $2$, or infinitely many ends (follows for locally finite graphs from Halin's paper \cite[Satz~2]{zbMATH03202996} and for the general case see \cite[Corollary 4]{zbMATH00446364}). It is a consequence of a result of Jung \cite{zbMATH03769659} that if such a graph has more than one end then there is a type 2 automorphism that fixes precisely two ends and thus the graph has at least two thin ends.  In particular, in the two-ended case both of the ends must be thin. Contrary to this, we deduce from Theorem 
\ref{thm:aut_G_inv_td} that the end of a one-ended transitive graph is always thick. This even holds in the more general case of quasi-transitive graphs.   This was proved for locally finite graphs by Thomassen \cite[Proposition~5.6]{zbMATH00130612}.  A variant of this result for \emph{metric ends} was proved by Kr\"on and M\"oller in \cite[Theorem~4.6]{zbMATH05349235}.

\begin{thm}
\label{thm:oneend-qt-thick}
If $G$ is a one-ended, quasi-transitive graph, then the unique end is thick.
\end{thm}

For the proof we need the following auxiliary result.

\begin{prop}
\label{prp:no-1end-qt-tree}
There is no one-ended quasi-transitive tree.
\end{prop}

\begin{proof}
Assume that $T$ is a quasi-transitive tree and that $R$ is a ray in $T$. Then there is an edge-orbit under $\Aut(T)$ containing infinitely many edges of $R$. Contract all edges not in this orbit to obtain a tree $T'$ whose automorphism group acts transitively on edges. Clearly, every end of $T'$ corresponds to an end of $T$ (there may be more ends of $T$ which we contracted). But edge transitive trees must be either regular, or bi-regular. Hence $T'$, and thus also $T$, has at least $2$ ends.
\end{proof}

\begin{proof}[Proof of Theorem~\ref{thm:oneend-qt-thick}]
Assume for a contradiction that $G$ is a quasi-transitive, one-ended graph whose end is thin. 

If the end $\omega$ is dominated, then remove all vertices which dominate it and only keep the component $C$ in which $\omega$ lies. The resulting graph is still quasi-transitive since $C$ must be stabilised setwise by every automorphism. Furthermore, the degree of $\omega$ does not increase by deleting parts of the graph. Hence we can without loss of generality assume that the end of the counterexample $G$ is undominated.

Now apply Theorem~\ref{thm:aut_G_inv_td} to $G$. This gives a nested set $\mathcal S$ of separations which is invariant under automorphisms---in particular, there are only finitely many orbits of $\mathcal S$ under the action of $\Aut(G)$. Theorem~\ref{thm:aut_G_inv_td} further tells us that there is a bijection between $\mathcal S$ and the edges of a one-ended tree $T$ such that the action of $\Aut(G)$ on $\mathcal S$ induces an action on $T$ by automorphisms. Hence $T$ is a quasi-transitive one-ended tree, which contradicts Proposition~\ref{prp:no-1end-qt-tree}.
\end{proof}

\appendix
\section{Appendix}

We say that a vertex $v$ \emph{dominates} a ray $L$ if there are infinitely many $v-L$ paths, any 
two only having $v$ as a common vertex.  It follows from the definition of an end that if a vertex 
domintes one ray belonging to an end then it dominates every ray belonging to that end and dominates the end. 

\begin{proof}[Proof of Lemma~\ref{lem:finite_dominating}] 
Assume that the set $X$ of dominating vertices is infinite.  By the above we can assume that there 
is a ray $R$ and infinitely many vertices $x_1, x_2, \ldots$ that dominate $R$ in $G$.  We show that 
$G$ must then contain a subdivision of the complete graph on $x_1, x_2, \ldots$.  Start by taking 
vertices $v_1$ and $v_2$ on $R_1$ such that there are disjoint $x_1-v_1$ and $x_2-v_2$ paths.  Then 
we find vertices $w_1$ and $w_2$ furher along the ray $R_1$ such that there are disjoint $x_1-w_1$ 
and $x_3-w_3$ paths and still further along we find vertices $u_2$ and $u_3$ such that there are 
disjoint $x_2-u_2$ and $x_3-u_3$ paths.  Adding the relevant segments of $R$ we find $x_1-x_2$, 
$x_1-x_3$ and $x_2-x_3$ paths having at most their endvertices in common.  The subgraph of $G$ 
consisting of these three paths is thus a subdivision of the complete graph on three vertices.    
Using induction we can find an increasing sequence of subgraphs $H_n$ of $G$ that contains the 
vertices $x_1, x_2, \ldots, x_n$ and also paths $P_{ij}$ linking $x_i$ and $x_j$ such that any two 
such paths have at most their end vertices in common.  The subgraph $H_n$ is a subdivision of the 
complete graph on $n$-vertices.  The subgraph $H=\bigcup_{i=1}^\infty H_i$ is a subdivision of the 
complete graph on (countably) infinite set of vertices and contains an infinite family of pairwise 
disjoint rays that all belong to the end $\omega$.  This contradicts our assumptions and we conclude 
that $T$ must be finite.  
\end{proof}

A \emph{ray decomposition}\footnote{Halin used the German term `schwach m-fach 
kettenf\"ormig'.} of \emph{adhesion $m$} of a graph $G$ consists of subgraphs $G_1, G_2, \ldots$ 
such that:
\begin{enumerate}
\item $G=\bigcup_{i=1}^\infty G_i$;
\item if $T_{n+1}=\big(\bigcup_{i=1}^n G_i\big)\cap G_{n+1}$ then $|T_{n+1}|=m$ and 
$T_{n+1}\subseteq G_n\setminus \big(\bigcup_{i=1}^{n-1} G_i\big)$ for $n=1, 2, \ldots$;
\item for each value of $n=1, 2, \ldots$ there are $m$ pairwise disjoint paths in $G_{n+1}$ that 
have  their initial vertices in $T_{n+1}$ and teminal vertices in $T_{n+2}$;
\item none of the subgraphs $G_i$ contains a ray.
\end{enumerate}

  The 
following Menger-type result  is used by Halin in his proof of \cite[Satz~2]{zbMATH03212734}.  In 
the proof we also use ideas from another one of Halin's papers \cite[Proof of Satz 
3]{zbMATH03205933}.

\begin{thm}  \label{thm:menger-like} Let $G$ be a locally finite connected graph with the property 
that $G$ contains a family of $m$ pairwise disjoint rays but there is no such family of $m+1$ 
pairwise disjoint rays.  Then there is in $G$ a family of pairwise disjoint separators $T_1, T_2, 
\ldots$ such that each contains precisely $m$ vertices and a ray in $G$ must for some $n_0$ 
intersects all the sets $T_n$ for $n\geq n_0$. 
\end{thm}

\begin{proof}  
Fix a reference vertex $v_0$ in $G$.  Let $E_j$ denote the set of vertices in distance precisely 
$j$ 
from $v_0$.  Define also $B_i$ as the set of vertices in distance at most $i$ from $v_0$.   For 
numbers $i$ and $j$ such that $i+1<j$ we construct a new graph $H_{ij}$ such that we start with the 
subgraph of $G$ induces by $B_j$, then we remove $B_i$ but add a new vertex $a$ that has as its 
neighbourhood the set $\partial B_i$ (for a set $C$ of vertices $\partial C$ denotes the set of vertices that are not in $C$ but are adjacent to some vertex in $C$) and we also add a new vertex $b$ that has every vertex in 
$\partial (G\setminus B_j)$ as its neighbour.  Since $G$ is assumed to be locally finite the graph 
$H_{ij}$ is finite.  
(By abuse of notation we do not distinguish the additional vertices $a$ and $b$ in different graphs 
$H_{ij}$.)

Suppose that, for a fixed value of $i$, there are always for $j$ big enough  at least $k$ distinct 
$a-b$ paths in $H_{ij}$ such that any two of them interesect only in the vertices $a$ and $b$.   
Then one can use the same argument as in the proof of K\"onig's Infinity Lemma to show that then 
$G$ 
contains a family of $k$ pairwise disjoint rays.  Because $G$ does not contain a family of $m+1$ 
pairwise disjoint rays there are for each $i$ a number $j_i$ such that for every $j\geq j_i$ there 
are at most $m$ disjoint $a-b$ paths in $H_{ij_i}$.     Since $a$ and $b$ are not adjacent in 
$H_{ij_i}$ then the Menger Theorem says that minimum number of a vertices in an $a-b$ separator is 
equal to the maximal number of $a-b$ paths such that any two of the paths have no inner vertices in 
common.  Whence there is in $H_{ij_i}\setminus\{a,b\}$ a set $T$ and $a-b$ separator with precisely 
$m$ vertices.  This set is also an separator in $G$ and every ray in $G$ that has its initial 
vertex 
in $B_i$ must intersect $T$.  From this information we can easily construct our sequence of 
separators $T_1, T_2, \ldots$. 

We can also clearly assume that if $i_j$ is the smallest number such that $T_j$ is in $B_{i_j}$ 
then 
$T_k\cap B_{i_j}=\emptyset$ for all $k>j$. 
\end{proof}

\begin{cor}
Let $G$ be a connected locally finite graph.  Suppose $\omega$ is an end of $G$ and $\omega$ has 
finite vertex degree $m$.  Then there is a sequence  $T_1, T_2, \ldots$ of separators each 
containing precisely $m$ vertices such that if $C_i$ denotes the component of $G- T_i$ that 
$\omega$ belongs to then $C_1\supseteq C_2\supseteq\ldots$ and $\bigcap_{i=1}^\infty 
C_i=\emptyset$. 
\end{cor}

\begin{proof}
We use exactly the same argument as above except that when we construct the $H_{ij}$ we only put in 
edges from $b$ to those vertices in $E_j$ that are in the boundary of the component of $G\setminus 
B_j$ that $\omega$ lies in.
\end{proof}

\begin{proof}[Proof of Lemma~\ref{Lrelevant}.]
The first part of the Lemma about the existence of a family of $k$ pairwise disjoint rays in 
$\omega$ with their initial vertices in $A\cap B$ follows directly from the above.

For the second part, the only thing we need to show is that there cannot exist a separation $(C,D)$ 
of order $<k$ such that $A\subseteq C$ and $\omega$ lies in $D$.  Such a separation cannot exist 
because the $k$ pairwise disjoint rays that have their initial vertices in $A\cap B$ and belong to 
$\omega$ would all have to pass through $C\cap D$.   
\end{proof}

\begin{thm} \label{Tchainshaped} {\rm(\cite[Satz~2]{zbMATH03212734})}  Let $G$ be a graph with the 
property that it contains a family of $m$ pairwise disjoint rays but no family of $m+1$ pairwise 
disjoint rays.  Let $X$ denote the set of vertices in $G$ that dominate some ray.  Then the set $X$ 
is finite and the graph $G- X$ has a ray decomposition of adhesion $m$.  
\end{thm} 
 
\begin{proof}  Let $R_1, \ldots, R_m$ denote a family of pairwise disjoint rays.  Set $R=R_1\cup 
\cdots\cup R_m$.  

Any ray in $G$ must intersect the set $R$ in infinitely many vertices and thus intersects one of 
the 
rays $R_1, \ldots, R_m$ in infinitely many vertices.  From this we conclude that every ray in $G$ 
is 
in the same end as one of the rays $R_1, \ldots, R_m$.  Thus a vertex that dominates some ray in 
$G$ 
must dominate one of the rays $R_1, \ldots, R_m$.
 
 In Lemma~\ref{lem:finite_dominating} we have already shown that the set of vertices dominating an 
end of finite vertex degree is finite.
Note also that if a vertex in $R$ is in infinitely many distinct sets of the type $\partial C$ 
where 
$C$ is a component of $G\setminus R$ then $x$ would be a dominating vertex of some ray $R_i$.  Thus 
there can only be finitely many vertices in $R$ with this property.  

We will now show that $G-X$ has a ray decomposition of adhesion $m$.  To simplify the 
notation we 
will in the rest of the proof assume that $X$ is empty.

Assume now that there is a component $C$ of $G-R$ such that $\partial C$ is infinite.  
Take 
a spanning tree of $C$ and then adjoin the vertices in $\partial C$ to this tree using edges in 
$G$. 
 Now we have a tree with infinitely many leafs.  It is now apparent that either the tree contains a 
ray that does not intersect $R$ or there is a vertex in $C$ that dominates a ray in $G$.  Both 
possibilities are contrary to our assumptions and we can conclude that $\partial C$ is finite for 
every component $C$ of $G\setminus R$.  

For every set $S$ in $R$ of such that  $S=\partial C$  for some component $C$ in $G\setminus R$ we 
find a locally finite connected subgraph $C_S$ of $C\cup S$ containing $S$.  The graph $G'$ that is 
the union of $R$ and all the subgraphs $C_S$ is a locally finite graph.  The original graph $G$ has 
a ray decomposition of adhesion $m$ if and only if $G'$ has a ray decomposition of adhesion $m$.  

At this point we apply Theorem~\ref{thm:menger-like}.
From Theorem~\ref{thm:menger-like} we have the sequence $T_2, T_3, \ldots$ of separators.   We 
choose $T_2$ such that all the rays $R_1, \ldots, R_m$ intersect $T_2$.    We start by defining 
$G_i$ for $i\geq 2$ as the union of $T_i$ and all those components of $G- T_i$ that contain 
the tail of some ray $R_i$.  Finally, set $G_1=G\setminus (G_2\setminus T_2)$.   Note that none of 
the subgraphs $G_i$ can contain a ray and our family of rays provides a family of $m$ pairwise 
disjoint $T_i-T_{i+1}$ paths.  Now we have shown that $G$ has a ray decomposition of adhesion $m$.
 \end{proof}

Finally, we are now ready to show how Halin's result above implies Theorem~\ref{Tsequence} that concerns 
$\omega$-relevant separations.

\begin{proof}[Proof of Theorem~\ref{Tsequence}.]  We continue with the notation in the proof of 
Theorem~\ref{Tchainshaped}.  Recall that there are infinitely many pairwise disjoint paths 
connecting a ray $R_i$ to a ray $R_j$.  Thus we may assume that the initial vertices of the rays 
$R_1, \ldots R_k$ all belong to the same component of $G-T_2$.  We set $A_n$ as the union of the 
component of $G-T_{n+1}$ that contains these initial vertices with $T_{n+1}$.  Then set $B_n=(G\sm 
A_n)\cup T_{n+1}$.  Now it is trivial to check that the sequence $(A_n, B_n)$ of separations 
satisfies the conditions.
\end{proof}

\section{Appendix}

\begin{proof}[Proof of Lemma~\ref{lem:treefixtail}.]
Let $\sigma$ be an automorphism of $T$.  
In cite \cite[Proposition~3.2]{zbMATH03341048} Tits proved that there are three types of 
automorphisms of a tree:  (i) those that fix some vertex, (ii) those that fix no vertex but leave 
an 
edge invariant and (iii) those that leave some double-ray $\ldots,, v_{-1}, v_0, v_1, v_2, \ldots$  
invariant and act as non-trivial translations on that double-ray.  (Similar results were proved 
independently by Halin in \cite{zbMATH03417497}.)   Since $T$ is one-ended it contains no double-ray and 
thus  (iii) is impossible.  Suppose now that $\sigma$ fixes no vertex in $T$ but leaves the edge 
$e$ 
invariant.  The end of $T$ lies in one of the components of $T-e$ and $\sigma$ swaps the two 
components of $T-e$.  This is impossible, because $T$ has only one end and this end must belong to 
one of the components of $T-e$.  Hence  $\sigma$ must fix some vertex $v$.  There is a unique ray 
$R'$ in $T$ with $v$ as an initial vertex and this ray is fixed pointwise by $\sigma$.  The two 
rays 
$R$ and $R'$ intersect in a ray that is a tail of $R$ and this tail of $R$ is fixed pointwise by 
$\sigma$.  
\end{proof}

\section{Appendix}

In this Appendix we prove Lemma~\ref{Lrayless} which is a slightly sharpened version of Lemma 3 
from 
Halin's paper \cite{zbMATH01701677}.  The change is that  \lq uncountable\rq\  in Halin's 
results is replaced by \lq at least 
$2^{\aleph_0}$ elements\rq.

First there is an auxilliary result that corresponds to Lemma 2 in \cite{zbMATH01701677}.

\begin{lem}\label{Lfiniteoruncountable}
Let $G$ be a connected graph and $\Gamma=\Aut(G)$.  Suppose $D$ is a subset of the vertex set of 
$G$.  Let $\{C_i\}_{i\in I}$ denote the family of components of $G-D$.  
Define $G_i$ as the subgraph spanned by $C_i\cup\partial C_i$.  Set 
$\Gamma_{i}=\Aut(G_{i})_{(\partial C_i)}$.  Suppose that $\Gamma_i$ is either finite or has at 
least $2^{\aleph_0}$ elements for all $i$.  Then $\Gamma_{(D)}$ is either 
finite or has at least $2^{\aleph_0}$ elements. 
\end{lem}

\begin{proof}     If one of the groups $\gamma_i$ has at least $2^{\aleph_0}$ elements then 
there is nothing more to do.  So, we assume that all these groups are finite.  

Now there are two situations where it is possible that $\Gamma_{(D)}$ is infinite.  The first is 
when infinitely many of the groups $\Gamma_i$ are non-trivial.
For any family $\{\sigma_i\}_{i\in I}$ such that $\sigma_i\in \Gamma_{i}$ we can find an 
automorphism $\sigma\in \Gamma_{(G\setminus C_i)}\subseteq \Gamma_{(D)}$ such that the restriction 
to $C_i$ equals $\sigma_i$ for all $i$.  If infinitely many of the groups $\Gamma_{C_i}$ are 
nontrivial, then there are at least $2^{\aleph_0}$ such families $\{\sigma_i\}_{i\in I}$ and 
$\Gamma_{(D)}$ must have at least $2^{\aleph_0}$ elements.  

We say that two components $C_i$ and $C_j$ are equivalent if $\partial C_i=\partial C_j$ and there 
is an isomorphism $\varphi_{ij}$ from the subgraph $G_i$ to the subgraph $G_j$ fixing every vertex 
in $\partial C_i=\partial C_j$.  Clearly there is an automorpism $\sigma_{ij}$ of $G$ that fixes 
every vertex that is neither in $C_i$ nor $C_j$ such that $\sigma_{ij}(v)=\varphi_{ij}(v)$ for 
$v\in 
C_i$ and $\sigma_{ij}(v)=\varphi_{ij}^{-1}(v)$  for $v\in C_j$.  If there are infinitely many 
disjoint ordered pairs of equivalent components we can for any subset of these pairs find an 
automorphism $\sigma\in\Gamma_{(D)}$ such that if $(C_i, C_j)$ is in our subset then the 
restriction 
of $\sigma$ to $C_i\cup C_j$ is equal to the restriction of $\sigma_{ij}$.  There are at least 
$2^{\aleph_0}$ such sets and thus $\Gamma_{(D)}$  has at least $2^{\aleph_0}$ elements.  

If neither of the two cases above occurs then   $\Gamma_{(D)}$ is clearly finite.
\end{proof}

\begin{proof}[Proof of Lemma~\ref{Lrayless}.]  
Following Schmidt \cite{zbMATH03834026} (see also  Halin's paper \cite[Section 3]{zbMATH01416675}) 
we define, using induction, for each ordinal $\lambda$ a class of graphs $A(\lambda)$.  The class 
$A(0)$ is the class of finite graphs.  Suppose $\lambda>0$ and $A(\mu)$ has already been defined 
for 
all $\mu<\lambda$.  A graph $G$ is in the class $A(\lambda)$ if and only if it contains a finite 
set 
$F$ of vertices such that each component of $G-F$ is in $A(\mu)$ for some $\mu<\lambda$.  It is 
shown in the papers referred to above that if $G$ belongs to $A(\lambda)$ for some ordinal $\lambda$ 
then 
$G$ is rayless and, conversely, every rayless graph belongs to $A(\lambda)$ for some ordinal 
$\lambda$.  For a rayless graph $G$ we define $o(G)$ as the smallest ordinal $\lambda$ such that 
$G$ 
is in $A(\lambda)$.  

The Lemma is proved by induction over $o(G)$.  If $o(G)=0$ then the graph $G$ is finite and the 
automorphism group is also finite.

Assume that the result is true for all rayless graphs $H$ such that $o(H)<o(G)$.  Find a finite set 
$F$ of vertices such that each of the components of $G-F$ has a smaller order than $G$.  Denote the 
family of components of $G-F$ with $\{C_i\}_{i\in I}$.  Denote with $G_i$ the subgraph induced by 
$C_i\cup\partial C_i$.  By induction hypothesis the pointwise stabiliser of $\partial C_i$ in 
$\Aut(G_i)$ is either finite or has at least $2^{\aleph_0}$ elements.   
Lemma~\ref{Lfiniteoruncountable} above implies that $\Aut(G)_{(D)}$ is either finite or has at 
least 
 $2^{\aleph_0}$ elements.
\end{proof}

\bibliographystyle{abbrv}
\bibliography{references}
\end{document}